\documentclass[preprint,12pt,nopreprintline]{elsarticle}

\usepackage{amsmath}
\usepackage{amsfonts}
\usepackage{amsthm}
\usepackage{lipsum}
\usepackage{mathtools}
\PassOptionsToPackage{cm}{fullpage}
\usepackage{mathrsfs}
\usepackage{hyperref}   
\usepackage{cleveref} 
\usepackage[cm]{fullpage}
\newtheorem{theorem}{Theorem}
\usepackage{xcolor} 
\usepackage{hyperref}

\hypersetup{
    colorlinks=true,       
    linkcolor=blue,       
    citecolor=blue,        
    urlcolor=blue,         
    allcolors=blue         
}

\makeatletter
\def\ps@pprintTitle{%
     \let\@oddhead\@empty
     \let\@evenhead\@empty
     \let\@oddfoot\@empty
     \let\@evenfoot\@oddfoot
}
\makeatother

\journal{Nuclear Physics B}

\begin{document}

\begin{frontmatter}

\title{On the generalized inverse tangent integral and Catalan’s constant}

\author[inst1]{Petr Vlachopulos}

\affiliation[inst1]{organization={Department of Mathematics and Statistics, Masaryk University, Faculty of Science},
            addressline={Kotlarska 2}, 
            city={Brno},
            postcode={611 37},
            country={Czech Republic}}

\begin{abstract}
In this paper, we develop new identities for the inverse tangent integral by connecting it to the dilogarithmic (polylogarithmic) structure and to a carefully designed auxiliary arctangent integral $Ti_2(a)$ with a tunable endpoint.

The core idea is based on the introduction of an auxiliary integral depending on two parameters and analyzing it via a generating-function perspective. This converts the integral into an explicit formula, yielding a compact representation in terms of the real part of a dilogarithmic expression plus a companion dilogarithm contribution. In parallel, the inverse tangent integral is rewritten through standard integral transformations into a form governed by the imaginary part of a dilogarithm evaluated at a complex argument, producing a clean polylogarithmic description.

Overall, we establish a coherent bridge between arctangent-type integrals, identities connected to Catalan’s constant, and the systematic generation of new integral representations and decompositions.

\end{abstract}


\end{frontmatter}

\label{sec:sample1}

\vspace{0.5cm}

\newcommand\restr[2]{{
  \left.\kern-\nulldelimiterspace 
  \littletaller 
  \right|_{#2} 
  }}

\newpage

\section{Introduction}

The inverse tangent integral
\begin{equation}
\operatorname{Ti}_2\left(y\right):=\int_0^y\frac{\operatorname{arctan}\left(x\right)}{x}\operatorname{dx}
\end{equation}
occupies a natural place at the intersection of special functions, polylogarithms, and the realm of classical constants. In particular, its evaluation for \(x=1\) gives the famous Catalan’s constant
\begin{equation}
G=\operatorname{Ti}_2\left(1\right),
\end{equation}
which can be considered one of the most prominent constants arising in analytic number theory and in the analysis of arctangent-type series (often infinite) and integrals. Precisely because of this connection, identities involving $\operatorname{Ti}_2$ are often of great interest, as they usually reveal deeper structural links between elementary trigonometric functions, higher transcendental functions, and many other subfields of analytic number theory \cite{lewin1958dilogarithms,campbell2022special}. From a historical perspective, it is almost obligatory to mention Ramanujan’s famous paper on the inverse tangent integral \cite{ramanujan1915integral}. This is the classical paper on the integral itself and related identities. Additional and more detailed work of Ramanujan on $\operatorname{Ti}_2$ can also be found in his first profound notebook \cite{berndt1985ramanujan}. 

In this work, we focus on the systematic development of similar links, as mentioned above, by showing how generalized inverse tangent integrals can be expressed through dilogarithmic, Fourier-type, and Hurwitz zeta-functional constructions.

A central theme of the paper is that arctangent integrals admit a richer analytic structure than is apparent from their original integral definitions. The starting point consists of the classical polylogarithmic identity relating $\operatorname{Ti}_2$ to the imaginary part of the dilogarithm, which already suggests that inverse tangent integrals should be accessible through complex-analytic methods, as one can see in \cite{maximon2003dilogarithm,wood1992computation}. Building on this idea, the paper introduces an auxiliary two-parameter integral with a tunable endpoint and studies it through differentiation with respect to the parameter and expansion to an explicit compact formula. As a result, the inverse tangent integral is placed into a broader framework in which trigonometric integrals, polylogarithms, and endpoint-dependent decompositions interact in a transparent way.

Within the first main result, we establish a representation of 
$\operatorname{Ti}_2\left(a\right)$ through the auxiliary integral $I\left(a,b\right)$, where the endpoint $b=b\left(a\right)$ is determined by a natural admissibility condition for $a$, which is a strictly positive number. This produces a nontrivial decomposition in which the inverse tangent integral is balanced by elementary logarithmic terms and real and imaginary parts of suitable combinations of dilogarithms. Moreover, by setting $a=1$, we recover the famous Catalan’s constant $G$ by obtaining the associated quantity $b$ from the auxiliary integral $I\left(1,b\right)$.

The significance of this construction lies not only in the final identity itself, but also in the method. In other words, the passage from an arctangent integral to a dilogarithmic representation is achieved through a carefully chosen generating-function mechanism, which may be viewed as analogous to other situations in special-function theory where orthogonal-polynomial expansions serve as a bridge between integral transforms and closed-form evaluations, as shown in \cite{ismail2002q,szego1975colloquium}. 

The second result develops a decomposition formula for the so-called inverse tangent integral
\begin{equation}
\int_0^A\frac{1}{x}\operatorname{arctan}\left(\frac{x}{\alpha}\right)\operatorname{dx},
\end{equation}
where  $x\in\mathbb{R}\setminus\{0\}$, $\alpha>0$, $\alpha\in\mathbb{R}\setminus\{k\pi\}$ and $A>0$. 

\newpage
We obtain such a decomposition formula via Euler’s product formula for the sine function and the corresponding partial fraction expansion of the cotangent. This leads to an identity in which the original inverse tangent integral can be decomposed into a hyperbolic inverse tangent-type contribution and an infinite family of correction terms. The underlying philosophy is reminiscent of spectral or mode expansions in a way that a single integral quantity is decomposed into a principal term together with an infinite series of structured contributions indexed by the poles of the cotangent \cite{cvijovic2009summation}. This type of representation is particularly effective, especially when the value of the endpoint is restricted to $A=1$. Here, Catalan’s constant emerges, and the correction terms can be further rewritten in terms of inverse tangent integrals, dilogarithmic values, and eventually Hurwitz zeta functions.

In this way, the paper contains more than isolated formulas, it demonstrates a coherent analytic pathway from simple trigonometric identities to deeper expansions into special functions. Such synthesis is one of the main conceptual strengths of the following results, as it places Catalan’s constant into a broader realm of identities that are governed by dilogarithmic and zeta-functional phenomena.

\section{Main results}

\begin{theorem}{\textbf{(Inverse tangent integral with tunable endpoint)}}\label{thm.1}
\textit{Assume that $a>0$ is a positive real number, $b=
b\left(a\right)\in\left(0,\pi\right)$ and $\operatorname{Li}_2$ is analytically continued dilogarithm on the principal branch. Let}
\begin{equation}
\nonumber
\psi\left(a\right):=\mathscr{Im}\left(\operatorname{Li}_2\left(1+ia\right)\right), \hspace{0,2cm} \phi_a\left(b\right):=I\left(a,b\right)-\frac{\pi b}{2}+\frac{b^2}{2},
\end{equation}
\textit{where}
\begin{equation}
\nonumber
I\left(a,b\right):=\int_0^b\operatorname{arctan}\left(\frac{a+\operatorname{cos}\left(\beta\right)}{\operatorname{sin}\left(\beta\right)}\right)\operatorname{d\beta}.
\end{equation}
\textit{Next, let us define the admissible set $\mathscr{A}:=\{a>0\hspace{0.1cm}\vert\hspace{0.1cm}0<\psi\left(a\right)<\phi_a\left(\pi\right)\}\neq\emptyset$. If $a\in\mathscr{A}$, then there exists a unique $b\in\left(0,\pi\right)$ such that}
\begin{equation}\label{eq:1}
\begin{aligned}
\operatorname{Ti_2}\left(a\right):=\int_0^a\frac{1}{x}\operatorname{arctan}\left(x\right)\operatorname{dx}=&\operatorname{arctan}\left(a\right)\operatorname{ln}\left(a\right)+I\left(a,b\right)-\frac{\pi b}{2}+\frac{b^2}{2}-\frac{\pi}{4}\operatorname{ln}\left(a^2+1\right).
\end{aligned}
\end{equation}
\end{theorem}

\vspace{0.2cm}
\newtheorem{corollary}{Corollary}
\begin{corollary}{\textbf{(Evaluation of the Catalan’s constant)}}\label{cor.1}
\textit{Let $a>0$, $0<b<\pi$ and $I\left(a,b\right)$, $\phi_a\left(b\right)$, $\psi\left(a\right)$ be given as in Theorem \ref{thm.1}. Then for $a=1$, we have}
\begin{equation}
\nonumber
I\left(1,b\right)=\frac{\pi b}{2}-\frac{b^2}{4},\hspace{0.2cm}\phi_1\left(b\right)=\frac{b^2}{4}
\end{equation}
\textit{with}
\begin{equation}
\nonumber
0<\psi\left(1\right)<\phi_1\left(\pi\right)=\frac{\pi^2}{4}.
\end{equation}
\textit{Moreover, there exists a unique $b\left(1\right)\in\left(0,\pi\right)$, namely}
\begin{equation}
\nonumber
b\left(1\right)=2\sqrt{\psi\left(1\right)}=\sqrt{4G+\pi\operatorname{ln}\left(2\right)}
\end{equation}
\textit{such that $\phi_1\left(b\left(1\right)\right)=\psi\left(1\right)$. Consequently, we obtain an evaluation of Catalan’s constant}
\begin{equation}
G=\frac{b^2\left(1\right)}{4}-\frac{\pi}{4}\operatorname{ln}\left(2\right).
\end{equation}
\end{corollary}

\vspace{0.2cm}
\begin{theorem}{\textbf{(Generalized inverse tangent integral)}}\label{thm.2}
\textit{Let $x\in\mathbb{R}\setminus\{0\}$, $\alpha>0$, $\alpha\in\mathbb{R}\setminus\{k\pi\}$, where $k\in\mathbb{Z}$ and $A>0$ is a positive real number, then}
\begin{equation}\label{eq:2}
\begin{aligned}
\operatorname{Ti}_2\left(\frac{A}{\alpha}\right)&:=\int_0^A\frac{1}{x}\operatorname{arctan}\left(\frac{x}{\alpha}\right)
\\&=\int_0^A\frac{1}{x}\operatorname{arctan}\left(\operatorname{cotg}\left(\alpha\right)\operatorname{tanh}\left(x\right)\right)+\sum_{k=1}^{\infty}\int_0^A\frac{1}{x}\operatorname{arctan}\left(\frac{2\alpha x}{x^2+\left(k\pi\right)^2-\alpha^2}\right).
\end{aligned}
\end{equation}
\end{theorem}

\vspace{0.2cm}
\begin{corollary}\label{cor.2}
\textit{Let $0<\alpha<\pi$ and denote $H\left(A,\alpha\right):=\int_0^A\frac{1}{x}\operatorname{arctan}\left(\operatorname{cotg}\left(\alpha\right)\operatorname{tanh}\left(x\right)\right)\operatorname{dx}$. Then}
\begin{equation}\label{eq.7}
\operatorname{Ti}_2\left(\frac{A}{\alpha}\right)=H\left(A,\alpha\right)+\sum_{k=1}^{\infty}\Bigg[\operatorname{Ti}_2\left(\frac{A}{k\pi-\alpha}\right)-\operatorname{Ti}_2\left(\frac{A}{k\pi+\alpha}\right) \Bigg].
\end{equation}
\end{corollary}

\vspace{0.2cm}
\begin{corollary}\label{cor.3}\textbf{(Catalan’s family)}
\textit{Let $n\geq 2$ and put $A=\alpha=\frac{\pi}{n}$ in Theorem \ref{thm.2}. Then}
\begin{equation}
G=\int_0^{\frac{\pi}{n}}\frac{1}{x}\operatorname{arctan}\left(\operatorname{cotg}\left(\frac{\pi}{n}\right)\operatorname{tanh}\left(x\right)\right)
\operatorname{dx}+\sum_{k=1}^{\infty}\Bigg[\operatorname{Ti}_2\left(\frac{1}{nk-1}\right)-\operatorname{Ti}_2\left(\frac{1}{1+nk}\right) \Bigg]
\end{equation}  
\end{corollary}

\vspace{0.2cm}
\newtheorem{rem}{Remark}
\begin{rem}\label{rem.1}
\textit{Let $n=2$ in Corollary \ref{cor.3}. Then}
\begin{equation}
G=\sum_{k=1}^{\infty}\Bigg[\operatorname{Ti}_2\left(\frac{1}{2k-1}\right)-\operatorname{Ti}_2\left(\frac{1}{2k+1}\right) \Bigg].
\end{equation}
\end{rem}

\vspace{0.2cm}
\begin{corollary}{\textbf{(Clausen’s reduction of $\operatorname{Ti}_2\left(\operatorname{tan}\left(\theta\right)\right)$)}}\label{cor.4}
\textit{Define the Clausen’s function as}
\begin{equation}
\nonumber
\operatorname{Cl}_2\left(\phi\right):=-\int_0^{\phi}\operatorname{ln}\left(2\operatorname{sin}\left(\frac{t}{2}\right)\right)\operatorname{dt}, \hspace{0.2cm} 0\leq\phi\leq 2\pi.
\end{equation}
\textit{Then for every $0<\theta<\frac{\pi}{2}$, we have}
\begin{equation}
\operatorname{Ti}_2\left(\operatorname{tan}\left(\theta\right)\right)=\theta\operatorname{ln}\left(\operatorname{tan}\left(\theta\right)\right)+\frac{1}{2}\operatorname{Cl}_2\left(2\theta\right)+\frac{1}{2}\operatorname{Cl}_2\left(\pi-2\theta\right).
\end{equation}
\end{corollary}

\vspace{0.2cm}
\newtheorem{lemma}{Lemma}
\begin{lemma}{\textbf{(Representation of Catalan’s constant)}}\label{lem.1}
Catalan's constant admits the following representation
\begin{equation}\label{eq:L1-main}
G
=
K(1)
+
(1-\cot(1))
+
\sum_{n=1}^{\infty}
\frac{(-1)^n}{(2n+1)^2}\,
\pi^{-(2n+1)}
\left[
\zeta\!\left(2n+1,1+\frac{1}{\pi}\right)
-
\zeta\!\left(2n+1,1-\frac{1}{\pi}\right)
\right].
\end{equation}
where
\begin{equation}\label{eq:K1-final-compressed}
K(1)
=
-\sum_{j=1}^{\infty}\frac{\sin(2j)}{j}\,\operatorname{Ei}(-2j)
+\pi\log\Gamma\!\left(\frac{1}{\pi}\right)
+\left(1-\frac{\pi}{2}\right)\log\left(\pi\right)
-\frac{\pi}{2}\log\!\left(\frac{\pi}{\sin \left(1\right
)}\right).
\end{equation}
\end{lemma}

\newpage
\section{Preliminaries}

Before we proceed to prove the main results, it is necessary to introduce certain objects, tools and identities that will be important for our subsequent considerations.
\subsection{\underline{Special functions and identities}}

\newtheorem{definition}{Definition}
\vspace{0.1cm}
\begin{definition}[\textbf{Inverse tangent integral}]\cite{ramanujan1915integral}
\textit{Let $y\in\mathbb{R}$, then the inverse tangent integral is defined as}
\begin{equation}
\nonumber
\operatorname{Ti}_2\left(y\right):=\int_0^y\frac{1}{x}\operatorname{arctan}\left(x\right)\operatorname{dx}.
\end{equation}
\end{definition}

\begin{definition}[\textbf{Catalan’s constant}]
\begin{equation}
\nonumber
G:=\int_0^1\frac{1}{x}\operatorname{arctan}\left(x\right)\operatorname{dx}.
\end{equation}
\end{definition}

Equivalently, it can be defined as a specific value of the Dirichlet beta function \cite{apostol2013introduction}
\begin{equation}
\nonumber
G:=\beta\left(2\right)=\sum_{n=0}^{\infty}\frac{\left(-1\right)^n}{\left(2n+1\right)^2}.
\end{equation}

We will also frequently use the dilogarithmic function.

\begin{definition}[\textbf{Dilogarithm}]\cite{zagier2007dilogarithm}
\textit{Let $z\in\mathbb{C}$, with $\vert z\vert\leq 1$. Then the dilogarithmic function is defined as}
\begin{equation}
\nonumber
\operatorname{Li}_2\left(z\right):=\sum_{n=1}^{\infty}\frac{z^n}{n^2}.
\end{equation}
\end{definition}

Additionally, we recall the key identity between $\operatorname{Ti}_2$ and $\operatorname{Li}_2$ that will be useful in many simplifications and identifications with the dilogarithm.

For real $x$, we have
\begin{equation}
\nonumber
\operatorname{Ti}_2\left(x\right)=\mathscr{Im}\left(\operatorname{Li}_2\left(ix\right)\right),
\end{equation}
where $\mathscr{Im}$ denotes the imaginary part \cite{maximon2003dilogarithm}.

\newpage

\vspace{0.2cm}
\subsection{\underline{Exponential integral}}

\vspace{0.1cm}
\begin{definition}[\textbf{Exponential integral}]
\textit{Let $x$ be a real non-zero number, then the exponential integral is given by}
\begin{equation}
\nonumber
\operatorname{Ei}\left(x\right):=-\int_{-x}^{\infty}\frac{e^{-t}}{t}\operatorname{dt}.
\end{equation}
\end{definition}

The exponential integral can also be expressed through the following identity, which will be especially important in the proof of Lemma \ref{lem.1}.

For any real $\xi>0$, we have
\begin{equation}
\nonumber
\int_0^1\frac{e^{-\xi x}-1}{x}\operatorname{dx}=\operatorname{Ei}\left(-\xi\right)-\gamma-\operatorname{ln}\left(\xi\right),
\end{equation}
where $\gamma$ represents the Euler-Mascheroni constant.

\subsection{\underline{Kummer Fourier series, Hurwitz zeta function and Clausen’s function}}\label{subsec:3.3}

\vspace{0.1cm}
Finally, we shall introduce the so-called Kummer Fourier series for the log-gamma function \cite{erdelyi1953higher} and the Hurwitz zeta function. Both will be used principally in the proof of Lemma \ref{lem.1}.

For any real $x$ such that $0<x<1$, the log-gamma function of $x$ satisfies
\begin{equation}
\nonumber
\operatorname{ln}\left(\Gamma\left(x\right)\right)=\frac{1}{2}\operatorname{ln}\left(2\pi\right)+\left(x-\frac{1}{2}\right)\operatorname{ln}\left(2\operatorname{sin}\left(\pi x\right)\right)+\frac{1}{\pi}\sum_{j=1}^{\infty}\frac{\operatorname{sin}\left(2\pi jx\right)}{x}\operatorname{ln}\left(x\right).
\end{equation}

\vspace{0.1cm}
\begin{definition}[\textbf{Hurwitz zeta function}]\cite{apostol2013introduction}
\textit{For any $z\in\mathbb{C}$ such that $\operatorname{Re}\left(z\right)>1$ and a constant $c\notin\{0,-1,-2,...\}$, the Hurwitz zeta function is defined as}
\begin{equation}
\nonumber
\zeta\left(z,c\right)=\sum_{k=0}^{\infty}\frac{1}{\left(k+c\right)^z}.
\end{equation}
\end{definition}

\begin{definition}[\textbf{Clausen’s function}]\cite{apostol2013introduction}
\textit{The Clausen’s function is given by the following integral}
\begin{equation}
\operatorname{Cl}_2\left(\phi\right):=-\int_0^{\phi}\operatorname{ln}\left(2\Bigg\vert\operatorname{sin}\left(\frac{t}{2}\right)\Bigg\vert\right)\operatorname{dt},
\end{equation}
\textit{where $0<\phi<2\pi$.}
\end{definition}

\newpage
\section{Proof of Theorem \ref{thm.1} and Corollary \ref{cor.1} - Inverse tangent integral with tunable endpoint and evaluation of Catalan’s constant}

In this section, we provide a complete proof of Theorem \ref{thm.1}. To avoid notational ambiguity, we shall write \(\mathscr{R}(z)\) and \(\mathscr{I}(z)\) for the real and
imaginary parts of a complex number \(z\). Moreover, $\operatorname{Log}$ will denote the principal branch of the logarithm, and
$\operatorname{Li}_2$ the analytically continued dilogarithm on its principal branch,
holomorphic on \(\mathbb{C}\setminus [1,\infty)\). We also use the standard derivative formula
\begin{equation}
\frac{d}{dz}\operatorname{Li}_2(z) = -\frac{\operatorname{Log}(1-z)}{z},
\qquad z\in \mathbb{C}\setminus [1,\infty),
\end{equation}
with the removable singularity at \(z=0\).

We begin with the auxiliary integral
\begin{equation}
I(a,b):=\int_0^b \arctan\!\left(\frac{a+\cos\left(\beta\right)}{\sin\left(\beta\right)}\right)\,d\beta,
\qquad a>0,\quad 0<b<\pi.
\end{equation}

Let us define the following function
\[
f(a,\beta):=\arctan\!\left(\frac{a+\cos\beta}{\sin\beta}\right),
\qquad a>0,\quad 0\le \beta\le b<\pi.
\]
Since \(0<b<\pi\), we have \(\sin\beta>0\) for every \(\beta\in(0,b]\), and the integrand itself is
continuous on the interval \([0,b]\), with
\begin{equation}
\lim_{\beta\to 0^+}\arctan\!\left(\frac{a+\cos\beta}{\sin\beta}\right)=\frac{\pi}{2}.
\end{equation}
Hence, the auxiliary integral \(I(a,b)\) is well defined.

Moreover, the function \(f\) is \(C^1\) in the variable \(a\), and its partial derivative corresponds to
\begin{equation}
\frac{\partial f}{\partial a}(a,\beta)
=
\frac{1}{1+\left(\frac{a+\cos\left(\beta\right)}{\sin\left(\beta\right)}\right)^2}\cdot \frac{1}{\sin\left(\beta\right)}
=
\frac{\sin\left(\beta\right)}{1+2a\cos\left(\beta\right)+a^2}.
\end{equation}
Because the denominator
\begin{equation}
1+2a\cos\left(\beta\right)+a^2=(a+\cos\left(\beta\right))^2+\sin^2\left(\beta\right)
\end{equation}
is strictly positive for all values \(a>0\) and all \(\beta\in[0,b]\), the partial derivative
\(\partial f/\partial a\) is continuous on every compact rectangle
\([a_0,a_1]\times [0,b]\subset (0,\infty)\times [0,b]\). Therefore, differentiation under
the integral sign is justified, and we directly obtain
\begin{equation}
\frac{\partial I}{\partial a}(a,b)
=
\int_0^b \frac{\sin\left(\beta\right)}{1+2a\cos\left(\beta\right)+a^2}\,d\beta.
\end{equation}
Now, let us perform the following substitution
\begin{equation}
u=\cos\left(\beta\right)\implies du=-\sin\left(\beta\right)\,d\beta.
\end{equation}
When \(\beta=0\), we have \(u=1\), and when \(\beta=b\), we have \(u=\cos \left(b\right)\). Thus
\begin{equation}
\frac{\partial I}{\partial a}(a,b)
=
\int_{\cos b}^{1}\frac{du}{1+a^2+2au}.
\end{equation}
This integral itself can be considered elementary
\begin{equation}
\int \frac{du}{1+a^2+2au}
=
\frac{1}{2a}\operatorname{Log}(1+a^2+2au),
\end{equation}
so
\begin{equation}
\frac{\partial I}{\partial a}(a,b)
=
\frac{1}{2a}
\left[
\operatorname{Log}(1+a^2+2au)
\right]_{u=\cos b}^{u=1}.
\end{equation}
Hence
\begin{equation}\label{eq.24}
\frac{\partial I}{\partial a}(a,b)
=
\frac{1}{2a}
\operatorname{Log}\!\left(\frac{(1+a)^2}{1+2a\cos b+a^2}\right).
\end{equation}

Next, we shall define the following function
\begin{equation}
F(a,b):=
\frac{\pi b}{2}-\frac{b^2}{2}
-\operatorname{Li}_2(-a)
+\mathscr{R}\!\bigl(\operatorname{Li}_2(-ae^{ib})\bigr).
\end{equation}
From here, we claim that \(\partial_a F(a,b)=\partial_a I(a,b)\). 

Indeed, since \(-a\in (-\infty,0)\subset \mathbb{C}\setminus [1,\infty)\), the derivative
formula for the dilogarithm \(\operatorname{Li}_2\) gives us
\begin{equation}
\frac{d}{da}\bigl[-\operatorname{Li}_2(-a)\bigr]
=
-\left(
-\frac{\operatorname{Log}(1-(-a))}{-a}
\right)(-1)
=
\frac{\operatorname{Log}(1+a)}{a}.
\end{equation}
Now focus on the second dilogarithmic term. For \(0<b<\pi\), the point \(-ae^{ib}\) never
lies on the branch cut \([1,\infty)\), so \(\operatorname{Li}_2(-ae^{ib})\) is holomorphic in \(a\), and
\begin{equation}\label{eq.27}
\frac{d}{da}\operatorname{Li}_2(-ae^{ib})
=
-\frac{\operatorname{Log}(1+ae^{ib})}{-ae^{ib}}\cdot (-e^{ib})
=
-\frac{\operatorname{Log}(1+ae^{ib})}{a}.
\end{equation}
Taking the real parts of both sides in (\ref{eq.27}) yields
\begin{equation}
\frac{d}{da}\mathscr{R}\!\bigl(\operatorname{Li}_2(-ae^{ib})\bigr)
=
-\frac{1}{a}\mathscr{R}\!\bigl(\operatorname{Log}(1+ae^{ib})\bigr).
\end{equation}
Since the real part of \(1+a^{ib}\) satisfies
\begin{equation}
\Re\!\bigl(\operatorname{Log}(1+ae^{ib})\bigr)=\operatorname{Log}|1+ae^{ib}|,
\end{equation}
and
\begin{equation}
|1+ae^{ib}|^2
=
(1+ae^{ib})(1+ae^{-ib})
=
1+2a\cos \left(b\right)+a^2,
\end{equation}
we obtain
\begin{equation}
\mathscr{R}\!\bigl(\operatorname{Log}(1+ae^{ib})\bigr)
=
\frac{1}{2}\operatorname{Log}(1+2a\cos b+a^2).
\end{equation}
Therefore, we can straightforwardly compute its derivative
\begin{equation}
\frac{d}{da}\mathscr{R}\!\bigl(\operatorname{Li}_2(-ae^{ib})\bigr)
=
-\frac{1}{2a}\operatorname{Log}(1+2a\cos b+a^2).
\end{equation}
Combining the last two calculations gives us
\begin{equation}\label{eq.33}
\frac{\partial F}{\partial a}(a,b)
=
\frac{\operatorname{Log}(1+a)}{a}
-\frac{1}{2a}\operatorname{Log}(1+2a\cos b+a^2)
=
\frac{1}{2a}
\operatorname{Log}\!\left(\frac{(1+a)^2}{1+2a\cos b+a^2}\right).
\end{equation}
Comparing the relations (\ref{eq.24}) and (\ref{eq.33}), we conclude that
\begin{equation}
\frac{\partial F}{\partial a}(a,b)=\frac{\partial I}{\partial a}(a,b).
\end{equation}

It remains to determine the constant of integration inside the auxiliary integral. Setting \(a=0\), we have
\begin{equation}
F(0,b)=\frac{\pi b}{2}-\frac{b^2}{2},
\end{equation}
since \(\operatorname{Li}_2(0)=0\). On the other hand, we have
\begin{equation}
I(0,b)
=
\int_0^b \arctan\!\left(\frac{\cos\left(\beta\right)}{\sin\left(\beta\right)}\right)\,d\beta
=
\int_0^b \arctan(\cot\left(\beta\right))\,d\beta.
\end{equation}
Moreover, for \(0<\beta<\pi\), we have \(\arctan(\cot\left(\beta\right))=\frac{\pi}{2}-\beta\), because \(\frac{\pi}{2}-\beta\in(-\pi/2,\pi/2)\) and also \(\tan\!\left(\frac{\pi}{2}-\beta\right)=\cot\left(\beta\right)\). Thus
\[
I(0,b)
=
\int_0^b \left(\frac{\pi}{2}-\beta\right)\,d\beta
=
\frac{\pi b}{2}-\frac{b^2}{2}
=
F(0,b).
\]
Since \(I\) and \(F\) have the same derivative with respect to the variable \(a\) and are identical at \(a=0\), it follows that \(I(a,b)=F(a,b)\) for all \(a>0\) and \(0<b<\pi\), as claimed.

Now we derive a formula for the inverse tangent integral in terms of
\(\mathscr{I}(\operatorname{Li}_2(1+ia))\). In fact, this is the precise quantity that must be matched to the auxiliary
integral. We begin by proving the following proposition.

\newtheorem{prop}{Proposition}
\begin{prop}
\phantomsection\label{my.prop}
For every \(a>0\), the inverse tangent integral satisfies
\begin{equation}
\operatorname{Ti}_2(a)
=
\arctan(a)\operatorname{Log}(a)
+
\mathscr{I}\!\bigl(\operatorname{Li}_2(1+ia)\bigr)
-\frac{\pi}{4}\operatorname{Log}(1+a^2).
\end{equation}
\end{prop}

\begin{proof}[Proof of Proposition \ref{my.prop}]\normalfont
Let us define the following quantity
\begin{equation}\label{eq.38}
H(a):=\arctan(a)\operatorname{Log}(a)
+
\mathscr{I}\!\bigl(\operatorname{Li}_2(1+ia)\bigr)
-\frac{\pi}{4}\operatorname{Log}(1+a^2),
\qquad a>0.
\end{equation}
We want to show that \(H'(a)=\arctan(a)/a\), which actually implies that
\begin{equation}
H(a)=\int_0^a \frac{\arctan x}{x}\,dx=\operatorname{Ti}_2(a),
\end{equation}
once the initial value is checked.

First, notice that
\begin{equation}\label{eq.40}
\frac{d}{da}\bigl(\arctan(a)\operatorname{Log}(a)\bigr)
=
\frac{\operatorname{Log}(a)}{1+a^2}+\frac{\arctan(a)}{a}.
\end{equation}

Next we compute the derivative of the dilogarithmic term in (\ref{eq.38}). Because \(1+ia\notin [1,\infty)\), then
for every \(a>0\), the derivative formula for \(\operatorname{Li}_2\) gives us
\begin{equation}
\frac{d}{da}\operatorname{Li}_2(1+ia)
=
-\frac{\operatorname{Log}(1-(1+ia))}{1+ia}\cdot i
=
-\frac{i\,\operatorname{Log}(-ia)}{1+ia}.
\end{equation}
For \(a>0\), the principal logarithm satisfies \(\operatorname{Log}(-ia)=\operatorname{Log}(a)-\frac{i\pi}{2}\). Therefore, we obtain the shift identity
\begin{equation}
-i\,\operatorname{Log}(-ia)
=
-i\operatorname{Log}(a)-\frac{\pi}{2}.
\end{equation}
Hence
\begin{equation}\label{eq.42}
\frac{d}{da}\operatorname{Li}_2(1+ia)
=
\frac{-\frac{\pi}{2}-i\operatorname{Log}(a)}{1+ia}.
\end{equation}
Multiply the numerator and the denominator of (\ref{eq.42}) by \(1-ia\) yields
\begin{equation}
\frac{d}{da}\operatorname{Li}_2(1+ia)
=
\frac{\bigl(-\frac{\pi}{2}-i\operatorname{Log}(a)\bigr)(1-ia)}{1+a^2}.
\end{equation}
Now expand the numerator in the following way
\begin{equation}\label{eq.44}
\bigl(-\tfrac{\pi}{2}-i\operatorname{Log}(a)\bigr)(1-ia)
=
-\frac{\pi}{2}
+
i\frac{a\pi}{2}
-
i\operatorname{Log}(a)
-
a\operatorname{Log}(a).
\end{equation}
Its imaginary part corresponds to \(\frac{a\pi}{2}-\operatorname{Log}(a)\).
Hence
\begin{equation}\label{eq.46}
\frac{d}{da}\mathscr{I}\!\bigl(\operatorname{Li}_2(1+ia)\bigr)
=
\mathscr{I}\!\left(\frac{d}{da}\operatorname{Li}_2(1+ia)\right)
=
\frac{\frac{a\pi}{2}-\operatorname{Log}(a)}{1+a^2}.
\end{equation}

Finally,
\begin{equation}\label{eq.47}
\frac{d}{da}\left(-\frac{\pi}{4}\operatorname{Log}(1+a^2)\right)
=
-\frac{\pi a}{2(1+a^2)}.
\end{equation}

Combining the relations (\ref{eq.40}), (\ref{eq.46}), and (\ref{eq.47}), we obtain
\begin{equation}
H'(a)
=
\left(\frac{\operatorname{Log}(a)}{1+a^2}+\frac{\arctan(a)}{a}\right)
+
\frac{\frac{a\pi}{2}-\operatorname{Log}(a)}{1+a^2}
-
\frac{\pi a}{2(1+a^2)}
=
\frac{\arctan(a)}{a}.
\end{equation}
Therefore
\begin{equation}
H(a)=\int_0^a \frac{\arctan x}{x}\,dx + C
\end{equation}
for some integration constant \(C\in\mathbb{R}\).

In order to determine the integration constant \(C\), we should let \(a\to 0^+\) in (\ref{eq.38}). We get
\[
\arctan(a)\operatorname{Log}(a)\to 0,
\qquad
\mathscr{I}(\operatorname{Li}_2(1+ia))\to \mathscr{I}(\operatorname{Li}_2(1))=0,
\qquad
\operatorname{Log}(1+a^2)\to 0.
\]
Hence \(H(a)\to 0\), which actually implies that \(\operatorname{Ti}_2(0)=0\). Thus \(C=0\), and the claimed formula naturally follows.

This completes the proof of Proposition \ref{my.prop}.
\end{proof}

In the following procedure, let us focus on the definition and behavior of the terms \(\psi(a)\) and \(\phi_a(b)\). Motivated by the preceding proposition, let us define
\begin{equation}
\psi(a):=\mathscr{I}\!\bigl(\operatorname{Li}_2(1+ia)\bigr),
\end{equation}
and, for fixed \(a>0\), we define
\begin{equation}
\phi_a(b):=I(a,b)-\frac{\pi b}{2}+\frac{b^2}{2},
\qquad 0\le b\le \pi.
\end{equation}\label{eq.52}
Using the closed form of \(I(a,b)\), we obtain
\begin{equation}\label{eq.52}
\phi_a(b)
=
-\operatorname{Li}_2(-a)+\mathscr{R}\!\bigl(\operatorname{Li}_2(-ae^{ib})\bigr),
\end{equation}
for every \(0<b<\pi\).

The identity we seek is
\begin{equation}
\phi_a(b)=\psi(a),
\tag{4.6}
\end{equation}
because then the formula for the inverse tangent integral \(\operatorname{Ti}_2(a)\) becomes
\begin{equation}\label{eq.53}
\operatorname{Ti}_2(a)
=
\arctan(a)\operatorname{Log}(a)+\phi_a(b)-\frac{\pi}{4}\operatorname{Log}(1+a^2)
=
\arctan(a)\operatorname{Log}(a)+I(a,b)-\frac{\pi b}{2}+\frac{b^2}{2}
-\frac{\pi}{4}\operatorname{Log}(1+a^2),
\end{equation}
which exactly represents the statement of Theorem \ref{thm.1}.

In other words, all that remains to prove is the fact that under the admissibility condition, there exists a
unique \(b\in(0,\pi)\) satisfying the equality (\ref{eq.52}).

Let us now focus on the existence and uniqueness of the endpoint \(b(a)\) by proving the following lemma. 

\begin{lemma}\label{lem.2}
Let
\begin{equation}
\psi(a):=\mathscr{I}\!\bigl(\operatorname{Li}_2(1+ia)\bigr),
\qquad
\phi_a(b):=I(a,b)-\frac{\pi b}{2}+\frac{b^2}{2},
\qquad 0\le b\le \pi.
\end{equation}
Assume that
\begin{equation}
0<\psi(a)<\phi_a(\pi)
=
\mathscr{R}\!\bigl(\operatorname{Li}_2(a)\bigr)-\operatorname{Li}_2(-a).
\end{equation}
Then there exists a unique \(b(a)\in(0,\pi)\) such that
\begin{equation}
\phi_a\bigl(b(a)\bigr)=\psi(a).
\end{equation}
\end{lemma}

\begin{proof}[Proof of Lemma \ref{lem.2}]\normalfont

In the upcoming proof, we proceed in several steps. First of all, we should examine the regularity of \(\phi_a\).

For fixed \(a>0\), the function \(b\mapsto I(a,b)\) is continuous on \([0,\pi]\), because
its integrand is continuous on \([0,\pi]\)
\begin{equation}
\lim_{\beta\to 0^+}\arctan\!\left(\frac{a+\cos\left(\beta\right)}{\sin\left(\beta\right)}\right)=\frac{\pi}{2},
\end{equation}
and
\begin{equation}
\lim_{\beta\to \pi^-}\arctan\!\left(\frac{a+\cos\left(\beta\right)}{\sin\left(\beta\right)}\right)
=
\begin{cases}
-\frac{\pi}{2}, & 0<a<1,\\[4pt]
0, & a=1,\\[4pt]
\frac{\pi}{2}, & a>1.
\end{cases}
\end{equation}
Therefore, the function \(\phi_a\) is continuous on \([0,\pi]\). Additionally, we can say that it is \(C^1\) on \((0,\pi)\) via the fundamental theorem of calculus.

Next, differentiating the auxiliary integral \(I(a,b)\) with respect to the upper limit \(b\), we obtain
\begin{equation}
\frac{\partial I}{\partial b}(a,b)
=
\arctan\!\left(\frac{a+\cos \left(b\right)}{\sin \left(b\right)}\right),
\qquad 0<b<\pi.
\end{equation}
Consequently,
\begin{equation}
\phi_a'(b)
= 
\arctan\!\left(\frac{a+\cos \left(b\right)}{\sin \left(b\right)}\right)-\frac{\pi}{2}+b.
\end{equation}

Now we shall rewrite the derivative \(\phi_a'(b)\) in a more transparent way that immediately implies strict monotonicity of \(\phi_a(b)\). Differentiating the relation (\ref{eq.52}) with respect to \(b\) provides us
\begin{equation}
\phi_a'(b)
=
\frac{d}{db}\mathscr{R}\!\bigl(\operatorname{Li}_2(-ae^{ib})\bigr).
\end{equation}
Since \(\frac{d}{db}(-ae^{ib})=-iae^{ib}\), the chain rule and the derivative formula for \(\operatorname{Li}_2\) give us
\begin{equation}
\frac{d}{db}\operatorname{Li}_2(-ae^{ib})
=
-\frac{\operatorname{Log}(1+ae^{ib})}{-ae^{ib}}\cdot (-iae^{ib})
=
-i\,\operatorname{Log}(1+ae^{ib}).
\end{equation}
Therefore
\begin{equation}
\phi_a'(b)
=
\mathscr{R}\!\bigl(-i\,\operatorname{Log}(1+ae^{ib})\bigr).
\end{equation}
Moreover, for any complex number \(w=u+iv\), we have \(-iw=v-iu\), so \(\mathscr{R}(-iw)=v=\Im(w)\). Hence
\begin{equation}
\phi_a'(b)
=
\mathscr{I}\!\bigl(\operatorname{Log}(1+ae^{ib})\bigr).
\end{equation}
Since \(\operatorname{Log}\) represents the principal logarithm, its imaginary part is the principal argument. Thus
\begin{equation}\label{eq.62}
\phi_a'(b)=\operatorname{Arg}(1+ae^{ib}).
\end{equation}

Next, for \(0<b<\pi\), we have \(\sin b>0\), so the imaginary part of \(1+ae^{ib}\) satisfies \(\mathscr{I}(1+ae^{ib})=a\sin \left(b\right)>0\). This implies that \(1+ae^{ib}\) lies strictly in the open upper half-plane. Therefore, its principal
argument satisfies
\begin{equation}
0<\operatorname{Arg}(1+ae^{ib})<\pi.
\end{equation}
Using the relation (\ref{eq.62}), we get
\begin{equation}
\phi_a'(b)>0,
\qquad 0<b<\pi.
\end{equation}
Thus \(\phi_a\) is strictly increasing on \((0,\pi)\).

In the next step, we should compute the endpoint values of \(\phi_a\left(b\right)\). At \(b=0\), we have \(I(a,0)=0\), so \(\phi_a(0)=0\).

We treat the other endpoint as a limit. Since the auxiliary integral \(I(a,b)\) is continuous up to \(b=\pi\), then the function \(\phi_a\) is also continuous up to \(b=\pi\). Moreover, from the relation (\ref{eq.52}), we get
\begin{equation}
\phi_a(\pi)
=
\lim_{b\longrightarrow\pi}\phi_a(b)
=
-\operatorname{Li}_2(-a)+\lim_{b\longrightarrow\pi}\mathscr{R}\!\bigl(\operatorname{Li}_2(-ae^{ib})\bigr).
\end{equation}
As \(b\longrightarrow\pi\), then \(-ae^{ib}\to a\) from one side of the branch cut, and the real part admits the limiting value \(\lim_{b\rightarrow\pi}\mathscr{R}\!\bigl(\operatorname{Li}_2(-ae^{ib})\bigr)=\mathscr{R}\!\bigl(\operatorname{Li}_2(a)\bigr)\). Hence
\begin{equation}
\phi_a(\pi)=\mathscr{R}\!\bigl(\operatorname{Li}_2(a)\bigr)-\operatorname{Li}_2(-a).
\end{equation}

Furthermore, since the function \(\phi_a\) is continuous on \([0,\pi]\), strictly increasing on \((0,\pi)\), and meets the endpoint conditions
\begin{equation}
\phi_a(0)=0,
\qquad
\phi_a(\pi)=\mathscr{R}\!\bigl(\operatorname{Li}_2(a)\bigr)-\operatorname{Li}_2(-a),
\end{equation}
the intermediate value theorem implies that the equation
\begin{equation}
\phi_a(b)=\psi(a)
\end{equation}
has a solution \(b\in(0,\pi)\) if and only if \(0<\psi(a)<\phi_a(\pi)\). Because \(\phi_a\) is strictly increasing, such a solution is automatically unique. 

This completes the proof of Lemma \ref{lem.2}. 
\end{proof}

In the above considerations, we have established the admissibility criterion/set, so we can proceed with the rest of the proof of Theorem \ref{thm.1}. Assume that \(a>0\) satisfies the admissibility condition
\begin{equation}
0<\psi(a)<\phi_a(\pi)
=
\mathscr{R}\!\bigl(\operatorname{Li}_2(a)\bigr)-\operatorname{Li}_2(-a).
\end{equation}
Using the result of Lemma \ref{lem.2}, there exists a unique \(b(a)\in(0,\pi)\) such that \(\phi_a\bigl(b(a)\bigr)=\psi(a)\). Applying the definition of \(\phi_a\), this equality gives us
\begin{equation}\label{eq.73}
I(a,b(a))-\frac{\pi b(a)}{2}+\frac{b(a)^2}{2}
=
\mathscr{I}\!\bigl(\operatorname{Li}_2(1+ia)\bigr).
\end{equation}
Now substitute (\ref{eq.73}) into the formula (\ref{eq.53}), we get
\begin{equation}\label{eq.75}
\operatorname{Ti}_2(a)
=
\arctan(a)\operatorname{Log}(a)
+
\mathscr{I}\!\bigl(\operatorname{Li}_2(1+ia)\bigr)
-\frac{\pi}{4}\operatorname{Log}(1+a^2).
\end{equation}
In other words
\begin{equation}
\operatorname{Ti}_2(a)
=
\arctan(a)\operatorname{Log}(a)
+
I(a,b(a))
-\frac{\pi b(a)}{2}
+\frac{b(a)^2}{2}
-\frac{\pi}{4}\operatorname{Log}(1+a^2).
\end{equation}
This completes the proof of Theorem \ref{thm.1}

The remaining part of this section contains the proof of Corollary \ref{cor.1}. Let us now specialize the result of Theorem \ref{thm.1} to \(a=1\).

For \(a=1\), the auxiliary integral simplifies considerably
\begin{equation}
I(1,b)=\int_0^b \arctan\!\left(\frac{1+\cos\left(\beta\right)}{\sin\left(\beta\right)}\right)\,d\beta.
\end{equation}
For \(0<\beta<\pi\), we have \(\sin\beta>0\) and we can rewrite the argument of \(\arctan\) as
\begin{equation}
\frac{1+\cos\left(\beta\right)}{\sin\left(\beta\right)}
=
\frac{2\cos^2(\beta/2)}{2\sin(\beta/2)\cos(\beta/2)}
=
\cot\!\left(\frac{\beta}{2}\right).
\end{equation}
Since \(0<\beta<\pi\), it directly follows that \(0<\frac{\beta}{2}<\frac{\pi}{2}\), hence
\begin{equation}
0<\frac{\pi}{2}-\frac{\beta}{2}<\frac{\pi}{2}.
\end{equation}
In other words, the term \(\frac{\pi}{2}-\frac{\beta}{2}\) lies in the principal range of \(\arctan\), and
\begin{equation}
\arctan\!\left(\cot\!\left(\frac{\beta}{2}\right)\right)
=
\frac{\pi}{2}-\frac{\beta}{2}.
\end{equation}
Thus
\begin{equation}\label{eq.80}
I(1,b)
=
\int_0^b \left(\frac{\pi}{2}-\frac{\beta}{2}\right)\,d\beta
=
\frac{\pi b}{2}-\frac{b^2}{4}.
\end{equation}
Consequently, putting the identity (\ref{eq.80}) back into \(\phi_1\left
(b\right)\) gives us
\begin{equation}
\phi_1(b)
=
I(1,b)-\frac{\pi b}{2}+\frac{b^2}{2}
=
\left(\frac{\pi b}{2}-\frac{b^2}{4}\right)-\frac{\pi b}{2}+\frac{b^2}{2}
=
\frac{b^2}{4}.
\end{equation}
Specifically, we have the following values at the endpoints
\begin{equation}
\phi_1(0)=0,
\qquad
\phi_1(\pi)=\frac{\pi^2}{4}.
\tag{4.15}
\end{equation}

It remains to verify the admissibility inequality \(0<\psi(1)<\phi_1(\pi)\). Recall that Catalan’s constant is defined via the following infinite series
\begin{equation}
G=\operatorname{Ti}_2(1)=\sum_{n=0}^\infty \frac{(-1)^n}{(2n+1)^2}.
\end{equation}
Since this is an alternating series with a positive first term, we have \(G>0\).
Moreover,
\begin{equation}
G
<
\sum_{n=0}^\infty \frac{1}{(2n+1)^2},
\end{equation}
because the right-hand side is obtained by removing the alternating signs and therefore
adding only positive contributions.

Now reindex this series as
\begin{equation}
\sum_{n=0}^\infty \frac{1}{(2n+1)^2}
=
\sum_{n=1}^\infty \frac{1}{(2n-1)^2}.
\end{equation}
The series of reciprocals of odd squares is obtained by subtracting the even-square part from the zeta function evaluated at \(z=2\),
\begin{equation}
\sum_{n=1}^\infty \frac{1}{(2n-1)^2}
=
\sum_{m=1}^\infty \frac{1}{m^2}
-
\sum_{n=1}^\infty \frac{1}{(2n)^2}.
\end{equation}
Since
\begin{equation}
\sum_{m=1}^\infty \frac{1}{m^2}=\frac{\pi^2}{6},
\qquad
\sum_{n=1}^\infty \frac{1}{(2n)^2}
=
\frac{1}{4}\sum_{n=1}^\infty \frac{1}{n^2}
=
\frac{\pi^2}{24},
\end{equation}
it follows that
\begin{equation}
\sum_{n=0}^\infty \frac{1}{(2n+1)^2}
=
\frac{\pi^2}{6}-\frac{\pi^2}{24}
=
\frac{\pi^2}{8}.
\end{equation}
This provides the estimation
\begin{equation}\label{eq.88}
0<G<\frac{\pi^2}{8}.
\end{equation}

Utilizing the relation (\ref{eq.75}) evaluated at \(a=1\) gives us
\begin{equation}
G=\operatorname{Ti}_2(1)=\mathscr{I}\!\bigl(\operatorname{Li}_2(1+i)\bigr)-\frac{\pi}{4}\operatorname{Log}\left(2\right).
\end{equation}
Therefore
\begin{equation}\label{eq.90}
\psi(1)=\mathscr{I}\!\bigl(\operatorname{Li}_2(1+i)\bigr)=G+\frac{\pi}{4}\operatorname{Log}\left(2\right).
\end{equation}
Since \(\operatorname{Log}\left(2\right)<\pi/2\), we have \(\frac{\pi}{4}\operatorname{Log}\left(2\right)<\frac{\pi^2}{8}\).
Combining this with (\ref{eq.88}), we obtain the desired admissibility condition
\begin{equation}
0<\psi(1)<\frac{\pi^2}{8}+\frac{\pi^2}{8}=\frac{\pi^2}{4}=\phi_1(\pi).
\tag{4.18}
\end{equation}
Therefore, the value \(a=1\) can be considered admissible, i.e. \(a\in\mathscr{A}\).

By Lemma \ref{lem.2}, there exists a unique \(b(1)\in(0,\pi)\) such that \(\phi_1(b(1))=\psi(1)\). Since \(\phi_1(b)=b^2/4\), this gives us
\begin{equation}
\frac{b(1)^2}{4}=\psi(1)\implies b(1)=2\sqrt{\psi(1)}.
\end{equation}
Substituting the relation (\ref{eq.90}) into this identity yields
\begin{equation}
b(1)=\sqrt{4G+\pi\operatorname{Log}\left(2\right)}.
\end{equation}
Finally, the result of Theorem \ref{thm.1} with \(a=1\) provides
\begin{equation}
G
=
\arctan(1)\operatorname{Log}(1)+\phi_1(b(1))-\frac{\pi}{4}\operatorname{Log}\left(2\right)
=
\phi_1(b(1))-\frac{\pi}{4}\operatorname{Log}\left(2\right).
\end{equation}
Since \(\phi_1(b)=b^2/4\), we can finally conclude that
\begin{equation}
G=\frac{b(1)^2}{4}-\frac{\pi}{4}\operatorname{Log}\left(2\right).
\end{equation}
This completes the proof of Corollary \ref{cor.1} and thus finishes this first section.

\section{Proof of Theorem \ref{thm.2} - Generalized inverse tangent integral}

In this section, we present the complete proof of Theorem \ref{thm.2}.

Let us start with Euler’s product formula
\begin{equation}\label{eq:32}
\operatorname{sin}\left(z\right)=z\prod_{k=1}^{\infty}\left(1-\frac{z^2}{\left(k\pi\right)^2}\right),
\end{equation}
where $z\in\mathbb{C}$. For any $z\neq \pi\mathbb{Z}$, we take the logarithm of (\ref{eq:32})
\begin{equation}\label{eq:33}
\operatorname{log}\left(\operatorname{sin}\left(z\right)\right)=\operatorname{log}\left(z\prod_{k=1}^{\infty}\left(1-\frac{z^2}{\left(k\pi\right)^2}\right)\right)=\operatorname{log}\left(z\right)+\sum_{k=1}^{\infty}\operatorname{log}\left(1-\frac{z^2}{\left(k\pi\right)^2}\right).
\end{equation}

\vspace{0.1cm}
Notice that on any compact set $K\in\mathbb{C}\setminus\pi\mathbb{Z}$, the infinite series (\ref{eq:33}) converges uniformly, since
\begin{equation}\label{eq:34}
\operatorname{log}\left(1-\frac{z^2}{\left(k\pi\right)^2}\right)\sim\operatorname{O}\left(\frac{1}{k^2}\right)
\end{equation}
uniformly on $K$. Hence, we can differentiate the relation (\ref{eq:33}) termwise in $K$. This provides us
\begin{equation}\label{eq:35}
\frac{d}{dz}\operatorname{log}\left(\operatorname{sin}\left(z\right)\right)=\frac{1}{z}+\sum_{k=1}^{\infty}\frac{d}{dz}\operatorname{log}\left(1-\frac{z^2}{\left(k\pi\right)^2}\right)=\frac{1}{z}-\sum_{k=1}^{\infty}\frac{2z}{\left(k\pi^2\right)-z^2}.
\end{equation}

Using the identity $\frac{d}{dz}\operatorname{log}\left(\operatorname{sin}\left(z\right)\right)=\operatorname{cotg}\left(z\right)$, we obtain the following expansion
\begin{equation}\label{eq:36}
\operatorname{cotg}\left(z\right)=\frac{1}{z}-\sum_{k=1}^{\infty}\frac{2z}{\left(k\pi^2\right)-z^2},
\end{equation}
where $z\neq\pi\mathbb{Z}$. Moreover, it is worth noting that on any compact set that does not include $\pi\mathbb{Z}$, the expansion (\ref{eq:36}) converges uniformly. This implies that we can substitute the variable $z$ and then take the real and imaginary parts separately.

Next, for $x\in\mathbb{R}$ and $\alpha\neq\pi\mathbb{Z}$, let us define the following function
\begin{equation}\label{eq:37}
\mathscr{F}\left(x\right):=\frac{\operatorname{sin}\left(\alpha+ix\right)}{\operatorname{sin}\left(\alpha\right)},
\end{equation}
Since $x$ is an arbitrary real number, the expression $\operatorname{sin}\left(\alpha+ix\right)$ is non-zero, because the zeros of $\operatorname{sin}\left(z\right)$ lie only on the real axis $z=\pi n$. Also $\operatorname{sin}\left(\alpha\right)\neq 0$ because $\alpha\notin\pi\mathbb{Z}$. 

This enables us to choose the argument $\operatorname{Arg}\left(\mathscr{F}\left(x\right)\right)$ as a continuous real function with a normalization condition $\operatorname{Arg}\left(\mathscr{F}\left(0\right)\right)=0$.

Using the classical identity $\operatorname{sin}\left(\alpha+ix\right)=\operatorname{sin}\left(\alpha\right)\operatorname{cosh}\left(x\right)+i\operatorname{cos}\left(\alpha\right)\operatorname{sinh}\left(x\right)$, we have
\begin{equation}\label{eq:38}
\mathscr{F}\left(x\right)=\operatorname{cosh}\left(x\right)+i\operatorname{cotg}\left(\alpha\right)\operatorname{sinh}\left(x\right).
\end{equation}

Since $\operatorname{cosh}\left(x\right)>0$ for all real $x$, the argument boils down to 
\begin{equation}\label{eq:39}
\operatorname{Arg}\left(\mathscr{F}\left(x\right)\right)=\operatorname{arctan}\left(\operatorname{cotg}\left(\alpha\right)\operatorname{tanh}\left(x\right)\right).
\end{equation}

At this point, we would like to differentiate the argument with respect to $x$ in order to express it as the imaginary part of $\operatorname{cotg}\left(\alpha+ix\right)$.

Let $I\subset\mathbb{R}$ be an interval, and let $G:I\longrightarrow\mathbb{C}\setminus\{0\}$ be a $C^1$ function such that $G\left(x\right)\neq 0$ for all points $x\in I$. Fix a point $x_0\in I$, since $G$ is continuous and $G\left(x_0\right)\neq 0$, there exists a subinterval $J\subset I$ containing $x_0$ such that $G\left(J\right)$ avoids a ray from the origin. In other words, it avoids a point on the unit circle after normalization. On this subinterval $J$, one can choose the so-called continuous branch of the argument 
\begin{equation}\label{eq:40}
\theta:J\longrightarrow\mathbb{R},
\end{equation}
such that
\begin{equation}\label{eq:41}
G\left(x\right)=\vert G\left(x\right)\vert e^{i\theta\left(x\right)}.
\end{equation}

The term $\theta$ represents what we mean by a continuous choice of $\operatorname{Arg}\left(\mathscr{F}\left(x\right)\right)$ on $J$, not necessarily the principal one.

We first prove the following proposition.
\begin{prop}\label{prop.1}
\textit{For all $x\in J$, we have}
\begin{equation}\label{eq:42}
\theta’\left(x\right)=\mathscr{Im}\left(\frac{G’\left(x\right)}{G\left(x\right)}\right)
\end{equation}
\end{prop}

\begin{proof}[Proof of Proposition~\ref{prop.1}]\normalfont
Since $G\left(x\right)$ is a non-zero function on $I$, we can write
\begin{equation}\label{eq:43}
G\left(x\right)=R\left(x\right)e^{i\theta\left(x\right)},
\end{equation}
where $R\left(x\right):=\vert G\left(x\right)\vert>0$. Since $G\in C^1$ and $R>0$, then $R,\theta$ are also $C^1$ on $J$.

Allow us to differentiate (\ref{eq:43})
\begin{equation}\label{eq:44}
G’\left(x\right)=R’\left(x\right)e^{i\theta\left(x\right)}+R\left(x\right)i\theta’\left(x\right)e^{i\theta\left(x\right)}=e^{i\theta\left(x\right)}\left(R’\left(x\right)+R\left(x\right)i\theta’\left(x\right)\right).
\end{equation}
Thus
\begin{equation}\label{eq:45}
\frac{G’\left(x\right)}{G\left(x\right)}=\frac{e^{i\theta\left(x\right)}\left(R’\left(x\right)+R\left(x\right)i\theta’\left(x\right)\right)}{R\left(x\right)e^{i\theta\left(x\right)}}=\frac{R’\left(x\right)}{R\left(x\right)}+i\theta’\left(x\right).
\end{equation}
At this point, it just remains to take imaginary parts of both sides of (\ref{eq:45})
\begin{equation}\label{eq:46}
\mathscr{Im}\left(\frac{G’\left(x\right)}{G\left(x\right)}\right)=\theta’\left(x\right).
\end{equation}
This completes the proof of Proposition \ref{prop.1}.
\end{proof}

From a topological perspective, it might be reasonable to ask why we can write $G\left(x\right)=\vert G\left(x\right)\vert e^{i\theta\left(x\right)}$ with $\theta$ continuous? To be clear, this is a classical topological fact. The map $u\left(x\right):=\frac{G\left(x\right)}{\vert G\left(x\right)\vert}$ is a continuous function from the subinterval $J$ to the unit circle $S^1$. Since $J$ is contractible, every continuous map $J\longrightarrow S^1$ admits a continuous lift to $\mathbb{R}$ through the covering map $t\mapsto e^{it}$. That lift precisely corresponds to $\theta\left(x\right)$.

Put another way, the argument $\theta\left(x\right)$ exists locally (also globally, since $I$ is simply connected), provided that $G$ is non-vanishing.

Now we have everything that we need to apply the identity (\ref{eq:42}) on $\theta\left(x\right):=\operatorname{Arg}\left(\mathscr{F}\left(x\right)\right)$. First, let us differentiate 
\begin{equation}\label{eq:47}
\frac{d}{dx}\operatorname{sin}\left(\alpha+ix\right)=i\operatorname{cos}\left(\alpha+ix\right).
\end{equation}
Thus
\begin{equation}\label{eq:48}
\frac{\mathscr{F}’\left(x\right)}{\mathscr{F}\left(x\right)}=\frac{i\operatorname{cos}\left(\alpha+ix\right)}{\operatorname{sin}\left(\alpha+ix\right)}=i\operatorname{cotg}\left(\alpha+ix\right).
\end{equation}
Therefore,
\begin{equation}\label{eq:49}
\frac{d}{dx}\operatorname{Arg}\left(\mathscr{F}\left(x\right)\right)=\mathscr{Im}\left(i\operatorname{cotg}\left(\alpha+ix\right)\right)=\mathscr{Re}\left(\operatorname{cotg}\left(\alpha+ix\right)\right).
\end{equation}

Recall the expansion (\ref{eq:36}) of $\operatorname{cotg}\left(z\right)$ and substitute $z=\alpha+ix$. We obtain
\begin{equation}\label{eq:50}
\operatorname{cotg}\left(\alpha+ix\right)=\frac{1}{\alpha+ix}+\sum_{k=1}^{\infty}\frac{2\left(\alpha+ix\right)}{\left(\alpha+ix\right)^2-\left(k\pi\right)^2}.
\end{equation}

Notice that the distance to the poles is positive, since $\mathscr{Im}\left(\alpha+ix\right)=x$ for $x\neq 0$ and even for $x=0$ we are away from the poles because $\alpha\notin\pi\mathbb{Z}$. This justifies the fact that we can take the real part of $\operatorname{cotg}\left(\alpha+ix\right)$ termwise on any strip $x\in[0,A]$, because the series (\ref{eq:50}) converges uniformly on such a strip.

Therefore, combining (\ref{eq:49}) and (\ref{eq:50}), we get
\begin{equation}\label{eq:51}
\frac{d}{dx}\operatorname{Arg}\left(\mathscr{F}\left(x\right)\right)=\mathscr{Re}\left(\frac{1}{\alpha+ix}\right)+\sum_{k=1}^{\infty}\mathscr{Re}\left(\frac{2\left(\alpha+ix\right)}{\left(\alpha+ix\right)^2-\left(k\pi\right)^2}\right).
\end{equation}
The first term gives us
\begin{equation}\label{eq:52}
\mathscr{Re}\left(\frac{1}{\alpha+ix}\right)=\mathscr{Re}\left({\frac{\alpha-ix}{\alpha^2+x^2}}\right)=\frac{\alpha}{\alpha^2+x^2}.
\end{equation}
Thus, integrating both sides of (\ref{eq:52}) from $0$ to $A$ yields
\begin{equation}\label{eq:53}
\int_0^A\mathscr{Re}\left(\frac{1}{\alpha+iw}\right)\operatorname{dw}=\int_0^A\frac{\alpha}{\alpha^2+w^2}\operatorname{dw}=\operatorname{arctan}\left(\frac{x}{\alpha}\right).
\end{equation}

In order to appropriately deal with the summand term in (\ref{eq:51}), we shall utilize a completely different strategy. Simply put, we will show that each summand is a derivative of the arctangent function with a specific argument. 

Fix $k\geq 1$ and define 
\begin{equation}\label{eq:54}
\Xi_k\left(x\right):=\operatorname{arctan}\left(\frac{2\alpha x}{x^2+\left(k\pi\right)^2-\alpha^2}\right).
\end{equation}
Notice that for $x=0$, we have $\Xi_k\left(0\right)=0$ and that $\Xi_k$ is a continuous function on $[0,\infty)$. This leads us to the following proposition.
\begin{prop}\label{prop.2}
\textit{Let $k\geq 1$ be a fixed number, then}
\begin{equation}\label{eq:55}
\frac{d}{dx}\Xi_k\left(x\right)=\mathscr{Re}\left(\frac{2\left(\alpha+ix\right)}{\left(\alpha+ix\right)^2-\left(k\pi\right)^2}\right),
\end{equation}
\textit{for all $x\geq 0$.}
\end{prop}
\begin{proof}[Proof of Proposition~\ref{prop.2}]\normalfont
Consider a complex function $\nu_k\left(x\right):=\left(k\pi\right)^2-\left(\alpha+ix\right)^2$,
then
\begin{equation}\label{eq:56}
\frac{d}{dx}\nu_k\left(x\right)=-2i\left(\alpha+ix\right).
\end{equation}
Hence, taking the logarithm of (\ref{eq:56}) and differentiating with respect to $x$ yields
\begin{equation}\label{eq:57}
\frac{d}{dx}\operatorname{log}\left(\nu_k\left(x\right)\right)=\frac{\nu’_k\left(x\right)}{\nu_k\left(x\right)}=\frac{-2i\left(\alpha+ix\right)}{\left(k\pi\right)^2-\left(\alpha+ix\right)^2}
\end{equation}

At this point, let us take the imaginary part of (\ref{eq:57}) and use the result of Proposition \ref{prop.1} with identification $G=\nu$. We get
\begin{equation}\label{eq:58}
\frac{d}{dx}\operatorname{Arg}\left(\nu_k\left(x\right)\right)=\mathscr{Im}\left(\frac{-2i\left(\alpha+ix\right)}{\left(k\pi\right)^2-\left(\alpha+ix\right)^2}\right).
\end{equation}
Simultaneously
\begin{equation}\label{eq:59}
\nu_k\left(x\right)=\left(x^2+\left(k\pi\right)^2-\alpha^2\right)-2i\alpha x.
\end{equation}

Hence, the argument of $\nu_k\left(x\right)$ must be
\begin{equation}\label{eq:60}
\operatorname{Arg}\left(\nu_k\left(x\right)\right)-\operatorname{arctan}\left(\frac{2\alpha x}{x^2+\left(k\pi\right)^2-\alpha^2}\right)=-\Xi_k\left(x\right)
\end{equation}
with the continuous choice $\operatorname{Arg}\left(\nu_k\left(0\right)\right)=0$.

Considering the right-hand side of (\ref{eq:58}), we can rewrite it as
\begin{equation}\label{eq:61}
\mathscr{Im}\left(\frac{-2i\left(\alpha+ix\right)}{\left(k\pi\right)^2-\left(\alpha+ix\right)^2}\right)=\mathscr{Re}\left(\frac{2\left(\alpha+ix\right)}{\left(k\pi\right)^2-\left(\alpha+ix\right)^2}\right)=-\mathscr{Re}\left(\frac{2\left(\alpha+ix\right)}{\left(\alpha+ix\right)^2-\left(k\pi\right)^2}\right).
\end{equation}
Hence
\begin{equation}\label{eq:62}
\frac{d}{dx}\Xi_k\left(x\right)=\mathscr{Re}\left(\frac{2\left(\alpha+ix\right)}{\left(\alpha+ix\right)^2-\left(k\pi\right)^2}\right).
\end{equation}

This completes the proof of Proposition \ref{prop.2}.
\end{proof}

Now it remains to insert the identity (\ref{eq:55}) into (\ref{eq:51}). We obtain
\begin{equation}\label{eq:63}
\frac{d}{dx}\operatorname{Arg}\left(\mathscr{F}\left(x\right)\right)=\frac{\alpha}{\alpha^2+x^2}-\sum_{k=1}^{\infty}\frac{d}{dx}\Xi_k\left(x\right).
\end{equation}
Let us integrate (\ref{eq:63}) from $0$ to $x\geq 0$. Since the infinite series in (\ref{eq:63}) is uniformly convergent on $[0,A]$ with $A>0$ (it decays as $\frac{1}{k^2}$), then for any fixed $A>0$, termwise integration gives us
\begin{equation}\label{eq:64}
\operatorname{Arg}\left(\mathscr{F}\left(x\right)\right)-\underbrace{Arg\left(\mathscr{F}\left(0\right)\right)}_{=0}=\operatorname{arctan}\left(\frac{x}{\alpha}\right)-\sum_{k=1}^{\infty}\left(\Xi_k\left(x\right)-\underbrace{\Xi_k\left(0\right)}_{=0}\right),
\end{equation}
\begin{equation}\label{eq:65}
\operatorname{Arg}\left(\mathscr{F}\left(x\right)\right)=\operatorname{arctan}\left(\frac{x}{\alpha}\right)-\sum_{k=1}^{\infty}\Xi_k\left(x\right),
\end{equation}
which is satisfied for all $x\geq 0$.

\newpage
Substitute the explicit form of $\operatorname{Arg}\left(\mathscr{F}\left(x\right)\right)$ from (\ref{eq:39}) to (\ref{eq:65}) and rearrange the terms as
\begin{equation}\label{eq:66}
\operatorname{arctan}\left(\frac{x}{\alpha}\right)=\operatorname{arctan}\left(\operatorname{cotg}\left(\alpha\right)\operatorname{tanh}\left(x\right)\right)+\sum_{k=1}^{\infty}\operatorname{arctan}\left(\frac{2\alpha x}{x^2+\left(k\pi\right)^2-\alpha^2}\right).
\end{equation}

This is exactly the key identity behind the formula in Theorem \ref{thm.2}. Now we simply turn it into the integral formula for any $A>0$ and $x\neq 0$
\begin{equation}\label{eq:67}
\int_0^A\frac{1}{x}\operatorname{arctan}\left(\frac{x}{\alpha}\right)=\int_0^A\frac{1}{x}\operatorname{arctan}\left(\operatorname{cotg}\left(\alpha\right)\operatorname{tanh}\left(x\right)\right)+\sum_{k=1}^{\infty}\int_0^A\frac{1}{x}\operatorname{arctan}\left(\frac{2\alpha x}{x^2+\left(k\pi\right)^2-\alpha^2}\right).
\end{equation}

This completes the proof of Theorem \ref{thm.2}.

\vspace{0.1cm}
To highlight the consistency of the result of Theorem \ref{thm.2}, we can easily show that each integral is well-defined, i.e. converges near $x=0$ and that the sum over all $k$’s is legitimately interchangeable with integration over $[0,A]$.

Let us first focus on the local integrability near $x=0$. The integrand on the left-hand side of (\ref{eq:67}) boils down to $\operatorname{arctan}\left(\frac{x}{\alpha}\right)\sim\frac{x}{\alpha}$ as $x\longrightarrow 0$. Hence
\begin{equation}\label{eq:68}
\frac{1}{x}\operatorname{arctan}\left(\frac{x}{\alpha}\right)\sim\frac{1}{\alpha},
\end{equation}
so the integral is bounded near zero.

The second term on the right-hand side of (\ref{eq:67}) behaves as
\begin{equation}\label{eq:69}
\operatorname{arctan}\left(\operatorname{cotg}\left(\alpha\right)\operatorname{tanh}\left(x\right)\right)\sim\operatorname{arctan}\left(\operatorname{cotg}\left(\alpha\right)x\right)\sim\operatorname{cotg}\left(\alpha\right)x,
\end{equation}
since $\operatorname{tanh}\left(x\right)\sim x$ as $x\longrightarrow 0$. Thus
\begin{equation}\label{eq:70}
\frac{1}{x}\operatorname{arctan}\left(\operatorname{cotg}\left(\alpha\right)\operatorname{tanh}\left(x\right)\right)\longrightarrow \operatorname{cotg\left(\alpha\right)}
\end{equation}
with $x\longrightarrow 0$, so it also converges.

For fixed $k$, the summand in (\ref{eq:67}) can be approximated in the following way
\begin{equation}\label{eq:71}
\operatorname{arctan}\left(\frac{2\alpha x}{x^2+\left(k\pi\right)^2-\alpha^2}\right)\sim\frac{2\alpha x}{\left(k\pi\right)^2-\alpha^2}
\end{equation}
as $x\longrightarrow 0$, so
\begin{equation}\label{eq:72}
\frac{1}{x}\operatorname{arctan}\left(\frac{2\alpha x}{x^2+\left(k\pi\right)^2-\alpha^2}\right)\sim\frac{2\alpha}{\left(k\pi\right)^2-\alpha^2}.
\end{equation}
This is finite because $\alpha\neq k\pi$, with $k\in\mathbb{Z}$. Hence, all of the integrals within the formula in Theorem \ref{thm.2} are well-defined. Additionally, the interchange of the summation and integration within the second term of (\ref{eq:67}) can be justified via the Weierstrass criterion, forming a test bound on $[0,A]$. 

\newpage
\section{Proofs of Corollaries \ref{cor.2}, \ref{cor.3}, \ref{cor.4} and Remark \ref{rem.1}}

Let us begin with the proof of Corollary \ref{cor.2}.

\begin{proof}[Proof of Corollary~\ref{cor.2}]\normalfont
Recalling the result of Theorem \ref{thm.2}, we know that
\begin{equation}
\operatorname{Ti}_2\left(\frac{A}{\alpha}\right)=\int_0^A\frac{1}{x}\operatorname{arctan}\left(\operatorname{cotg}\left(\alpha\right)\operatorname{tanh}\left(x\right)\right)+\sum_{k=1}^{\infty}\int_0^A\frac{1}{x}\operatorname{arctan}\left(\frac{2\alpha x}{x^2+\left(k\pi\right)^2-\alpha^2}\right).
\end{equation}
It remains to simplify each integral term for every k. Hence let us fix $k\geq 1$. Since $0<\alpha<\pi$, we have $k\pi-\alpha>0$ and $k\pi+\alpha>0$. So for every $x\geq 0$, one can define the following functions
\begin{equation}
u_k\left(x\right):=\operatorname{arctan}\left(\frac{x}{k\pi-\alpha}\right)\in [0,\frac{\pi}{2}),\hspace{0.2cm} v_k\left(x\right):=\operatorname{arctan}\left(\frac{x}{k\pi+\alpha}\right)\in[0,\frac{\pi}{2}).
\end{equation}
Additionally,
\begin{equation}
\frac{x}{k\pi-\alpha}\geq\frac{x}{k\pi+\alpha},
\end{equation}
thus
\begin{equation}
0\leq u_k\left(x\right)-v_k\left(x\right)<\frac{\pi}{2}.
\end{equation}
Therefore, the principal branch of the function $\operatorname{arctan}$ is compatible with the tangent subtraction formula in the following way
\begin{equation}\label{eq.115}
u_k\left(x\right)-v_k\left(x\right)=\operatorname{arctan}\left(\frac{\frac{x}{k\pi-\alpha}-\frac{x}{k\pi+\alpha}}{1+\frac{x^2}{\left(k\pi-\alpha\right)\left(k\pi+\alpha\right)}}\right).
\end{equation}
The argument of (\ref{eq.115}) can be simplified as
\begin{equation}
\frac{\frac{x}{k\pi-\alpha}-\frac{x}{k\pi+\alpha}}{1+\frac{x^2}{\left(k\pi-\alpha\right)\left(k\pi+\alpha\right)}}=\frac{x^2}{x^2+\left(k\pi\right)^2-\alpha^2}.
\end{equation}
Hence
\begin{equation}\label{eq.117}
\operatorname{arctan}\left(\frac{x^2}{x^2+\left(k\pi\right)^2-\alpha^2}\right)=\operatorname{arctan}\left(\frac{x}{k\pi-\alpha}\right)-\operatorname{arctan}\left(\frac{x}{k\pi+\alpha}\right).
\end{equation}
Dividing (\ref{eq.117}) by $x$ and integrating from $0$ to $A$ yields
\begin{equation}
\int_0^A\operatorname{arctan}\left(\frac{x^2}{x^2+\left(k\pi\right)^2-\alpha^2}\right)\operatorname{dx}=\int_0^A\frac{1}{x}\operatorname{arctan}\left(\frac{x}{k\pi-\alpha}\right)\operatorname{dx}-\int_0^A\operatorname{arctan}\left(\frac{x}{k\pi+\alpha}\right)\operatorname{dx}.
\end{equation}
By making a substitution $x=u\left(k\pi\mp\alpha\right)$, both integrals correspond to $\operatorname{Ti}_2\left(\frac{A}{k\pi-\alpha}\right)$ and $\operatorname{Ti}_2\left(\frac{A}{k\pi+\alpha}\right)$. 
Substituting this back to the formula in Theorem \ref{thm.2} completes the proof of Corollary \ref{cor.2}.
\end{proof}

\newpage
Let us now proceed to the proof of Corollary \ref{cor.3} which is a direct consequence of Corollary \ref{cor.2}.

\begin{proof}[Proof of Corollary~\ref{cor.3}]\normalfont
Put $A=\alpha=\frac{\pi}{n}$ with $n\geq 2$. Then
\begin{equation}
\frac{A}{\alpha}=1,
\end{equation}
so the left-hand side of (\ref{eq.7}) is precisely $\operatorname{Ti}_2\left(1\right)=G$. Furthermore,
\begin{equation}
\frac{A}{k\pi-\alpha}=\frac{\frac{\pi}{n}}{k\pi-\frac{\pi}{n}}=\frac{1}{nk-1},\hspace{0.2cm}\frac{A}{k\pi+\alpha}=\frac{\frac{\pi}{n}}{k\pi+\frac{\pi}{n}}=\frac{1}{nk+1}.
\end{equation}
Substituting this into Corollary \ref{cor.2} yields the desired result.
\end{proof}

Notice that the first few values of $n$ are especially nice
\begin{itemize}
    \item $n=2$:
    \begin{equation}
    G=\sum_{k=1}^{\infty}\Bigg[\operatorname{Ti}_2\left(\frac{1}{2k-1}\right)-\operatorname{Ti}_2\left(\frac{1}{2k+1} \right) \Bigg],
    \end{equation}
\end{itemize}
since the principal hyperbolic term vanishes. This is the formula in Remark \ref{rem.1}.

In summary, rational angle $\alpha=\frac{\pi}{n}$ generates a Catalan decomposition, i.e. different integro-summation representations of Catalan’s constant G, but only the value $\alpha=\frac{\pi}{2}$ is unique, because it eliminates the principal hyperbolic term. 

Now we shall proceed to the proof of Corollary \ref{cor.4}.

\begin{proof}[Proof of Corollary~\ref{cor.4}]\normalfont
Let us start from the defining inverse tangent integral and substitute $\operatorname{tan}\left(\theta\right)$ as an argument. Then
\begin{equation}
\operatorname{Ti}_2\left(\operatorname{tan}\left(\theta\right)\right)=\int_0^{\operatorname{tan}\left(\theta\right)}\frac{\operatorname{arctan}\left(x\right)}{x}\operatorname{dx}.
\end{equation}
Now, make a substitution $x=\operatorname{tan}\left
(u\right)$, where $u\in[0,\theta]$. Hence $\operatorname{dx}=\operatorname{sec}^2\left(u\right)\operatorname{du}$ and $\operatorname{arctan}\left(x\right)=u$. This gives us
\begin{equation}
\operatorname{Ti}_2\left(\operatorname{tan}\left(\theta\right)\right)=\int_0^{\theta}\frac{u}{\operatorname{tan}\left(u\right)}\operatorname{sec}^2\left(u\right)\operatorname{du}=\int_0^{\theta}\frac{u}{\operatorname{sin}\left(u\right)\operatorname{cos}\left(u\right)}\operatorname{du}.
\end{equation}
Notice that it is possible to utilize the logarithmic derivative in the following way
\begin{equation}
\frac{d}{du}\operatorname{ln}\left(\operatorname{tan}\left(u\right)\right)=\frac{\operatorname{sec}^2\left(u\right)}{\operatorname{tan}\left(u\right)}=\frac{1}{\operatorname{sin}\left(u\right)\operatorname{cos}\left(u\right)},
\end{equation}
thus
\begin{equation}\label{eq.125}
\operatorname{Ti}_2\left(\operatorname{tan}\left(u\right)\right)=\int_0^{\theta}u\operatorname{d}\left(\operatorname{ln}\left(\operatorname{tan}\left(u\right)\right)\right)=\theta\operatorname{ln}\left(\operatorname{tan}\left(\theta\right)\right)-\int_0^{\theta}\operatorname{ln}\left(\operatorname{tan}\left(u\right)\right)\operatorname{du},
\end{equation}
where the second equality comes from the integration by parts. 

Next, let us write
\begin{equation}
\operatorname{ln}\left(\operatorname{tan}\left(u\right)\right)=\operatorname{ln}\left(2\operatorname{sin}\left(u\right)\right)-\operatorname{ln}\left(2\operatorname{cos}\left (u\right)\right),
\end{equation}
hence 
\begin{equation}
\int_0^{\theta}\operatorname{ln}\left(\operatorname{tan}\left(u\right)\right)\operatorname{du}=\int_0^{\theta}\operatorname{ln}\left(2\operatorname{sin}\left(u\right)\right)\operatorname{du}-\int_0^{\theta}\operatorname{ln}\left(2\operatorname{cos}\left(u\right)\right)\operatorname{du}.
\end{equation}
Now we will compute both integrals separately. 

The definition of the Clausen’s function with $t=2u$ gives us
\begin{equation}
\operatorname{Cl}_2\left(2\theta\right)=-\int_0^{2\theta}\operatorname{ln}\left(2\operatorname{sin}\left(\frac{t}{2}\right)\right)\operatorname{dt}=-2\int_0^{\theta}\operatorname{ln}\left(2\operatorname{sin}\left(u\right)\right)\operatorname{du},
\end{equation}
so
\begin{equation}
\int_0^{\theta}\operatorname{ln}\left(2\operatorname{sin}\left(u\right)\right)\operatorname{du}=-\frac{1}{2}\operatorname{Cl}_2\left(2\theta\right).
\end{equation}
Again, by the definition of the Clausen’s function but with $t=\pi-2u$, we get
\begin{equation}
\operatorname{Cl}_2\left(\pi-2\theta\right)=-\int_0^{\pi-2\theta}\operatorname{ln}\left(2\operatorname{sin}\left(\frac{t}{2}\right)\right)\operatorname{dt}.
\end{equation}
Since
\begin{equation}
\operatorname{sin}\left(\frac{\pi-2u}{2}\right)=\operatorname{sin}\left(\frac{\pi}{2}-u\right)=\operatorname{cos}\left(u\right),
\end{equation}
we obtain the following expression
\begin{equation}
\int_0^{\theta}\operatorname{ln}\left(2\operatorname{cos}\left(u\right)\right)\operatorname{du}=\frac{1}{2}\operatorname{Cl}_2\left(\pi-2\theta\right).
\end{equation}
Therefore
\begin{equation}
\int_0^{\theta}\operatorname{ln}\left(\operatorname{tan}\left(u\right)\right)\operatorname{du}=-\frac{1}{2}\operatorname{Cl}_2\left(2\theta\right)-\frac{1}{2}\operatorname{Cl}_2\left(\pi-2\theta\right).
\end{equation}
Substituting this back into (\ref{eq.125}) yields the desired formula
\begin{equation}\label{eq.135}
\operatorname{Ti}_2\left(\operatorname{tan}\left(\theta\right)\right)=\theta\operatorname{ln}\left(\operatorname{tan}\left(\theta\right)\right)+\frac{1}{2}\operatorname{Cl}_2\left(2\theta\right)+\frac{1}{2}\operatorname{Cl}_2\left(\pi-2\theta\right).
\end{equation}

This completes the proof of Corollary \ref{cor.4}.
\end{proof}

After a quick observation, one can see that there are several further interesting values of $\theta$ and $A$ that produce elegant simplifications of \ref{eq.135}. For example, it is worth trying the following pairs: $\theta=\frac{\pi}{6}$ with $A=\frac{\alpha}{\sqrt{3}}$, $\theta=\frac{\pi}{3}$ with $A=\alpha\sqrt{3}$, $\theta=\frac{\pi}{8}$ with $A=\alpha\left(\sqrt{2}-1\right)$, or $\theta=\frac{\pi}{12}$ with $A=\alpha\left(2-\sqrt{3}\right)$.

\section{Proof of Lemma \ref{lem.1} - Representation of the Catalan’s constant}

In this section, we will focus on the simplification of the result of Theorem \ref{thm.2} which occurs at $A=1$. Each $k$-summand of the integral of the second term in (\ref{eq:67}) can be evaluated explicitly in terms of the inverse tangent integral $\operatorname{Ti}_2$, equivalently in terms of dilogarithms. Then the whole series can be directly expanded via the Hurwitz zeta function and the Gamma function.  

\begin{proof}[Proof of Lemma~\ref{lem.1}]\normalfont

From the decomposition formula obtained earlier in the paper, specialized to $\alpha=1$ and $A=1$, we have
\begin{equation}\label{eq:G-decomposition-start}
G
=
\int_0^1 \frac{\arctan(\cot(1)\tanh x)}{x}\,dx
+
\sum_{k=1}^{\infty}
\left[
\operatorname{Ti}_2\!\left(\frac{1}{k\pi-1}\right)
-
\operatorname{Ti}_2\!\left(\frac{1}{k\pi+1}\right)
\right].
\end{equation}

Now, let us denote the first term in (\ref{eq:G-decomposition-start}) as $K\left(1\right)$. Then we get

\begin{equation}\label{eq:G-K1-series}
G
=
K(1)
+
\sum_{k=1}^{\infty}
\left[
\operatorname{Ti}_2\!\left(\frac{1}{k\pi-1}\right)
-
\operatorname{Ti}_2\!\left(\frac{1}{k\pi+1}\right)
\right].
\end{equation}

In order to make the representation more efficient, it is desirable to expand the left-hand side of (\ref{eq:G-K1-series}) into an infinite series, which subsequently leads to the so-called Hurwitz zeta function. In general, the behavior of such a class of functions is well described.

Let us recall that for $|v|\le 1$,
\begin{equation}\label{eq:Ti2-power}
\operatorname{Ti}_2(v)=\sum_{n=0}^{\infty}\frac{(-1)^n}{(2n+1)^2}v^{2n+1}.
\end{equation}
For every integer $k\ge 1$, both numbers $(k\pi+1)^{-1}$ and $(k\pi-1)^{-1}$ are contained in the interval $(0,1)$, hence \eqref{eq:Ti2-power} may be applied to each term of the series in \eqref{eq:G-K1-series}. Therefore,
\[
\operatorname{Ti}_2\!\left(\frac{1}{k\pi\pm 1}\right)
=
\sum_{n=0}^{\infty}
\frac{(-1)^n}{(2n+1)^2}\frac{1}{(k\pi\pm 1)^{2n+1}}.
\]
Substituting this into \eqref{eq:G-K1-series}, we obtain
\begin{equation}\label{eq:G-double-series}
G
=
K(1)
+
\sum_{k=1}^{\infty}
\sum_{n=0}^{\infty}
\frac{(-1)^n}{(2n+1)^2}
\left[
\frac{1}{(k\pi+1)^{2n+1}}
-
\frac{1}{(k\pi-1)^{2n+1}}
\right].
\end{equation}

We now justify the interchange of the two summations. For fixed $n\ge 0$, the $k$-series is absolutely convergent since its summand is asymptotically $O(k^{-(2n+1)})$. For $n\ge 1$, the exponent satisfies $2n+1\ge 3$, and therefore the sum over $k$ is absolutely convergent with uniform and summable majorant. 

The remaining case $n=0$ must be handled separately, but even there, the difference
\[
\frac{1}{k\pi+1}-\frac{1}{k\pi-1}
=
-\frac{2}{(k\pi)^2-1}
\]
is asymptotically $O(k^{-2})$. Hence, it is summable over $k$. This implies that the double series is absolutely summable after taking the difference, and Fubini's theorem allows us to interchange the order of summation. Thus
\begin{equation}\label{eq:G-series-reordered}
G
=
K(1)
+
\sum_{n=0}^{\infty}\frac{(-1)^n}{(2n+1)^2}S_{2n+1},
\end{equation}
where
\begin{equation}\label{eq:Sr-def}
S_r
:=
\sum_{k=1}^{\infty}
\left(
\frac{1}{(k\pi-1)^r}
-
\frac{1}{(k\pi+1)^r}
\right).
\end{equation}

At this point, we have to evaluate the coefficient $S_r$. First, we consider the case when $r=1$. Then
\begin{equation}\label{eq:S1-expand}
S_1
=
\sum_{k=1}^{\infty}
\left(
\frac{1}{k\pi-1}
-
\frac{1}{k\pi+1}
\right)
=
\sum_{k=1}^{\infty}
\frac{2}{(k\pi)^2-1}.
\end{equation}
Now it is desirable to use the classical partial fraction identity
\begin{equation}\label{eq:cot-partial-fraction}
\sum_{k=1}^{\infty}\frac{1}{(k\pi)^2-b^2}
=
\frac{1}{2b^2}
-
\frac{\cot \left(b\right)}{2b},
\qquad b\notin \pi\mathbb{Z}.
\end{equation}
Setting up $b=1$ gives us
\[
\sum_{k=1}^{\infty}\frac{1}{(k\pi)^2-1}
=
\frac12-\frac{\cot \left(1\right)}{2},
\]
so
\begin{equation}\label{eq:S1-value}
S_1
=
2\left(\frac12-\frac{\cot \left(1\right)}{2}\right)
=1-\cot \left(1\right)
\end{equation}

Now let $r>1$. Then we have
\begin{align}
\sum_{k=1}^{\infty}\frac{1}{(k\pi+1)^r}
&=
\pi^{-r}\sum_{k=1}^{\infty}\frac{1}{\left(k+\frac{1}{\pi}\right)^r}
=
\pi^{-r}\zeta\!\left(r,1+\frac{1}{\pi}\right),
\label{eq:hurwitz-plus}
\\
\sum_{k=1}^{\infty}\frac{1}{(k\pi-1)^r}
&=
\pi^{-r}\sum_{k=1}^{\infty}\frac{1}{\left(k-\frac{1}{\pi}\right)^r}
=
\pi^{-r}\zeta\!\left(r,1-\frac{1}{\pi}\right).
\label{eq:hurwitz-minus}
\end{align}
Subtracting \eqref{eq:hurwitz-minus} from \eqref{eq:hurwitz-plus} gives us
\begin{equation}\label{eq:Sr-hurwitz}
S_r
=
\pi^{-r}
\left[
\zeta\!\left(r,1-\frac{1}{\pi}\right)
-
\zeta\!\left(r,1+\frac{1}{\pi}\right)
\right],
\qquad r>1.
\end{equation}

Directly substituting \eqref{eq:S1-value} and \eqref{eq:Sr-hurwitz} into \eqref{eq:G-series-reordered}, and separating the $n=0$ term yields
\begin{align}
G
&=
K(1)
+
(1-\cot \left(1\right))
+
\sum_{n=1}^{\infty}
\frac{(-1)^n}{(2n+1)^2}\,
\pi^{-(2n+1)}
\left[
\zeta\!\left(2n+1,1+\frac{1}{\pi}\right)
-
\zeta\!\left(2n+1,1-\frac{1}{\pi}\right)
\right].
\label{eq:G-final-hurwitz}
\end{align}
This establishes the anticipated formula \eqref{eq:L1-main}.

\vspace{0.2cm}
In what follows, we will derive the Fourier-type expansion formula for the integrand inside $K(1)$. This will eventually complete the whole proof. 

We can directly evaluate the auxiliary integral $K(1)$. For general $\alpha$ with $0<\alpha<\pi$, consider the term $\arctan(\cot\alpha \tanh x)$, where $x>0$.
Using the standard complex logarithm representation of the inverse tangent as
\begin{equation}\label{eq:atan-log}
\arctan \left(y\right)
=
\frac{1}{2i}\log\!\left(\frac{1+iy}{1-iy}\right),
\end{equation}
we obtain the following identity
\begin{equation}\label{eq:atan-general-alpha}
\arctan(\cot\left(\alpha\right)\,\tanh \left(x\right))
=
\frac{1}{2i}
\log\!\left(
\frac{1+i\cot\left(\alpha\right)\,\tanh \left
(x\right)}{1-i\cot\left
(\alpha\right)\tanh \left(x\right)
}\right).
\end{equation}

Now, set
\[
q:=e^{-2x},
\qquad 0<q<1\],
where $x>0$. Since
\[
\tanh x=\frac{1-q}{1+q},
\]
we may rewrite the quotient in \eqref{eq:atan-general-alpha} as
\[
\frac{1+i\cot\left(\alpha\right)\tanh \left(x\right)}{1-i\cot\left(\alpha\right)\tanh \left
(x\right)}
=
\frac{(1+q)+i\cot\left(\alpha\right)(1-q)}{(1+q)-i\cot\left(\alpha\right)(1-q)}.
\]
Multiplying the numerator and the denominator by $\sin\alpha$, and using Euler’s formula $\cos\alpha+i\sin\left(\alpha\right)=e^{i\alpha}$, gives us the following simplification
\begin{equation}\label{eq:quotient-simplified}
\frac{1+i\cot\left(\alpha\right)\tanh \left(x\right)}{1-i\cot\left(\alpha\right)\tanh \left(x\right)}
=
e^{\,i(\pi-2\alpha)}
\frac{1-qe^{2i\alpha}}{1-qe^{-2i\alpha}}.
\end{equation}
Substituting this back into (\ref{eq:atan-general-alpha}) yields
\begin{align}
\arctan\left(\cot\left(\alpha\right)\tanh \left(x\right)\right)
&=
\frac{1}{2i}\log\!\left(
e^{\,i(\pi-2\alpha)}
\frac{1-qe^{2i\alpha}}{1-qe^{-2i\alpha}}
\right)
\nonumber\\
&=
\frac{\pi}{2}-\alpha
+
\frac{1}{2i}
\left[
\log(1-qe^{2i\alpha})-\log(1-qe^{-2i\alpha})
\right].
\label{eq:atan-alpha-expanded}
\end{align}

Because $|q|<1$, the power series fro logarithm
\[
\log(1-z)=-\sum_{j=1}^{\infty}\frac{z^j}{j},
\qquad |z|<1,
\]
may be applied to both logarithms in the relation \eqref{eq:atan-alpha-expanded}. This gives us
\begin{align}
\arctan\left(\cot\left(\alpha\right)\tanh \left(x\right)\right)
&=
\frac{\pi}{2}-\alpha
+
\frac{1}{2i}
\left[
-\sum_{j=1}^{\infty}\frac{q^j e^{2ij\alpha}}{j}
+
\sum_{j=1}^{\infty}\frac{q^j e^{-2ij\alpha}}{j}
\right]
\nonumber\\
&=
\frac{\pi}{2}-\alpha
-
\sum_{j=1}^{\infty}\frac{\sin(2j\alpha)}{j}q^j
\nonumber\\
&=
\frac{\pi}{2}-\alpha
-
\sum_{j=1}^{\infty}\frac{\sin(2j\alpha)}{j}e^{-2jx}.
\label{eq:atan-alpha-fourier}
\end{align}

Let us now recall the classical Fourier series identity
\begin{equation}\label{eq:Fourier-sawtooth}
\sum_{j=1}^{\infty}\frac{\sin(2j\alpha)}{j}
=
\frac{\pi}{2}-\alpha,
\qquad 0<\alpha<\pi.
\end{equation}
Subtracting \eqref{eq:Fourier-sawtooth} from \eqref{eq:atan-alpha-fourier} provides us with the formula
\begin{equation}\label{eq:atan-alpha-vanishing}
\arctan\left(\cot\left(\alpha\right)\tanh \left(x\right)\right)
=
-\sum_{j=1}^{\infty}\frac{\sin(2j\alpha)}{j}\bigl(e^{-2jx}-1\bigr).
\end{equation}
This expression is especially convenient because the right-hand side vanishes at $x=0$, as it must, since the left-hand side tends to zero as $x\longrightarrow 0$.

\vspace{0.2cm}
Now we specialize the relation \eqref{eq:atan-alpha-vanishing} to $\alpha=1$ and perform termwise integration. This will eventually lead to the appearance of the exponential integral.

We have
\[
\arctan\left(\cot \left(1\right)\tanh \left(x\right)\right)
=
-\sum_{j=1}^{\infty}\frac{\sin(2j)}{j}\bigl(e^{-2jx}-1\bigr).
\]
Dividing this by $x$ and integrating from $0$ to $1$ formally gives us
\begin{equation}\label{eq:K1-formal}
K(1)
=
-\sum_{j=1}^{\infty}\frac{\sin(2j)}{j}
\int_0^1 \frac{e^{-2jx}-1}{x}\,dx.
\end{equation}

From a technical standpoint, we need to justify the interchange between the sum and the integral. For each fixed $j$, we have
\[
\left|\frac{e^{-2jx}-1}{x}\right|
\le
\min\{2j,1/x\}.
\]
Therefore
\[
\int_0^1 \left|\frac{e^{-2jx}-1}{x}\right|dx
\ll 1+\log \left(j\right).
\]

However, this still does not justify absolute summability, because after multiplying by \(\vert\operatorname{sin}\left(2j\right)\vert/2j\), the series 
\begin{equation}
\sum_{j=0}^{\infty}\frac{1+\operatorname{log}\left(j\right)}{j}
\end{equation}
diverges.

Therefore, we are forced to exploit the oscillations of \(\operatorname{sin}\left(2j\right)\). This leads us to the following proposition.
\begin{prop}\label{prop.4}
Let us define the following function
\begin{equation}
f\left(x\right):=\sum_{j=0}^{\infty}\frac{\operatorname{sin}\left(2j\right)}{2j}\left(1-e^{-2jx}\right).
\end{equation}
Next, let 
\begin{equation}
u_j\left(x\right):=\frac{1-e^{-2jx}}{j},\quad A_n:=\sum_{j=0}^{\infty}\operatorname{sin}\left(2j\right),\quad f_N\left(x\right):=\sum_{j=0}^{\infty}\operatorname{sin}\left(2j\right)u_j\left(x\right).
\end{equation}
Then \(f_N\left(x\right)\longrightarrow f\left(x\right)\) for every \(x\in(0,1]\).
\end{prop}

\vspace{0.1cm}
\begin{proof}[Proof of Proposition \ref{prop.4}]\normalfont
First, we shall focus on boundedness of the trigonometric partial sums. We compute \(A_n\) using the standard identity
\begin{equation}
\sum_{j=0}^{\infty}\operatorname{sin}\left(2j\right)=\frac{\operatorname{sin\left(n\right)\operatorname{sin}\left(n+1\right)}}{\operatorname{sin}\left(1\right)}.
\end{equation}
Hence 
\begin{equation}\label{eq.185}
\vert A_n\vert\leq\frac{1}{\operatorname{sin}\left(1\right)}:=M,\quad n\geq 1.
\end{equation}
Therefore, the partial sums of \(\operatorname{sin}\left(2j\right)\) are uniformly bounded by this \(M\).

Proceeding with the monotonicity of \(u_j\) in \(j\), we fix \(x\in[0,1]\) and consider the continuous extension
\begin{equation}
\gamma_x\left(t\right):=\frac{1-e^{2xt}}{t},
\end{equation}
where \(t>0\). Differentiating both sides with respect to \(t\) gives us
\begin{equation}
\gamma’_x\left(t\right)=\frac{2xte^{-2xt}-\left(1-e^{-2xt}\right)}{t^2}=\frac{e^{-2xt}\left(1+2xt\right)-1}{t^2}.
\end{equation}

Generally, for every \(y>0\), we have the following inequality
\begin{equation}
e^y>1+y\iff e^{-y}\left(1+y\right)<1.
\end{equation}

Setting \(y=2xt\), we obtain
\begin{equation}
e^{-2xt}\left(1+2xt\right)-1<0,
\end{equation}
so \(\gamma_x’\left(t\right)<0\).

Therefore, for each fixed \(x\in[0,1]\), the sequence \(u_j\left(x\right)=\gamma_x\left(j\right)\) is decreasing in \(j\). Additionally, we have
\begin{equation}\label{eq.190}
0\leq u_j\left(x\right)\leq u_1\left(j\right)=1-e^{-2x}\leq 2x
\end{equation}
and also
\begin{equation}
\sup_{x\in[0,1]}u_j\left(x\right)=\sup{x\in[0,1]}\frac{1-e^{-2jx}}{j}\leq\frac{1}{j}\longrightarrow 0.
\end{equation}

At this stage, it just remains to apply to perform the Abel summation for the partial sums. Since \(A_j=\sum_{k=1}^j\operatorname{sin}\left(2k\right)\), we get
\begin{equation}
f_N\left(x\right)=A_N u_N\left(x\right)+\sum_{j=1}^{N-1}A_j\left(u_j\left(x\right)-u_{j+1}\left(x\right)\right).
\end{equation}
Using the estimation (\ref{eq.185}) and the fact that \(u_j\left(x\right)\) decreases in \(j\), we directly obtain
\begin{equation}
\vert f_N\left(x\right)\vert\leq M u_N\left(x\right)+M\sum_{j=1}^{N-1}\left(u_j\left(x\right)-u_{j+1}\left(x\right)\right).
\end{equation}
One can observe that the sum telescopes as
\begin{equation}
\sum_{j=1}^{N-1}\left(u_j\left(x\right)-u_{j+1}\left(x\right)\right)=u_1\left(x\right)-u_N\left(x\right).
\end{equation}
Thus
\begin{equation}
\vert f_N\left(x\right)\vert\leq M u_N\left(x\right)+M\left(u_1\left(x\right)-u_N\left(x\right)\right)=M u_1\left(x\right).
\end{equation}

Using the estimation (\ref{eq.190}) provides us with
\begin{equation}
\vert f_N\left(x\right)\vert\leq M\left(1-e^{-2x}\right)\leq 2Mx.
\end{equation}
Therefore,
\begin{equation}
\Big\vert\frac{f_N\left(x\right)}{x}\Big\vert\leq 2M,
\end{equation}
for all \(0<x\leq 1\) and \(N\geq 1\). In other words, \(f_N\left(x\right)\longrightarrow f\left(x\right)\) pointwise for \(x\in(0,1]\).

This completes the proof of Proposition \ref{prop.4}.
\end{proof}

Now, applying the dominated convergence theorem to the result of Proposition \ref{prop.4} gives us
\begin{equation}
\int_0^1\frac{f\left(x\right)}{x}\operatorname{dx}=\lim_{N\longrightarrow\infty}\int_0^1\frac{f_N\left(x\right)}{x}\operatorname{dx}.
\end{equation}
Simultaneously, each \(f_N\) is a finite sum, so
\begin{equation}
\int_0^1\frac{f\left(x\right)}{x}\operatorname{dx}=\sum_{j=1}^N\operatorname{sin}\left(2j\right)\int_0^1\frac{u_j\left(x\right)}{x}\operatorname{dx}=\sum_{j=1}^N\frac{\operatorname{sin}\left(2j\right)}{j}\int_0^1\frac{1-e^{-2jx}}{x}\operatorname{dx}.
\end{equation}
Passing to the limit in \(N\) gives us
\begin{equation}\label{eq.200}
\begin{aligned}
K\left(1\right):=\int_0^1\frac{\operatorname{arctan\left(\operatorname{cotg\left(1\right)\operatorname{tanh}\left(x\right)}\right)}}{x}\operatorname{dx}=\sum_{j=1}^{\infty}\frac{\operatorname{sin}\left(2j\right)}{j}\int_0^1\frac{1-e^{-2jx}}{x}\operatorname{dx}= \\
=-\sum_{j=1}^{\infty}\frac{\operatorname{sin}\left(2j\right)}{j}\int_0^1\frac{e^{-2jx}-1}{x}\operatorname{dx}.
\end{aligned}
\end{equation}

Now, for $\xi>0$, let us define the following integral,
\begin{equation}\label{eq:T-xi-def}
T(\xi):=\int_0^1 \frac{e^{-\xi x}-1}{x}\,dx.
\end{equation}
Differentiating under the integral sign gives us
\begin{equation}\label{eq:Tprime}
T'(\xi)
=
-\int_0^1 e^{-\xi x}\,dx
=
-\frac{1-e^{-\xi}}{\xi}.
\end{equation}
On the other hand, we can write
\[
\frac{d}{d\xi}\bigl(\operatorname{Ei}(-\xi)-\gamma-\log \left(\xi\right)\bigr)
=
\frac{e^{-\xi}}{\xi}-\frac{1}{\xi}
=
-\frac{1-e^{-\xi}}{\xi},
\]
where $\gamma$ is the Euler-Mascheroni constant and $\operatorname{Ei}$ represents the Exponential integral.

Thus $T'(\xi)$ coincides with the derivative of the term $\operatorname{Ei}(-\xi)-\gamma-\log\left(\xi\right)$. Since both expressions go to $0$ as $\xi\to 0^+$, we can conclude that
\begin{equation}\label{eq:T-xi-closed}
T(\xi)=\operatorname{Ei}(-\xi)-\gamma-\log\left(\xi\right).
\end{equation}
Putting $\xi=2j$ in (\ref{eq:T-xi-def}) and applying it to \eqref{eq.200} gives us
\begin{equation}\label{eq:K1-after-Ei}
K(1)
=
-\sum_{j=1}^{\infty}\frac{\sin(2j)}{j}\bigl(\operatorname{Ei}(-2j)-\gamma-\log(2j)\bigr).
\end{equation}
Expanding the right-hand side leads us to
\begin{align}
K(1)
&=
-\sum_{j=1}^{\infty}\frac{\sin(2j)}{j}\operatorname{Ei}(-2j)
+
(\gamma+\log\left( 2\right))\sum_{j=1}^{\infty}\frac{\sin(2j)}{j}
+
\sum_{j=1}^{\infty}\frac{\sin(2j)}{j}\log \left(j\right).
\label{eq:K1-expanded}
\end{align}
Using the relation \eqref{eq:Fourier-sawtooth} with $\alpha=1$ gives us
\[
\sum_{j=1}^{\infty}\frac{\sin(2j)}{j}
=
\frac{\pi}{2}-1.
\]
Therefore
\begin{equation}\label{eq:K1-before-Kummer}
K(1)
=
-\sum_{j=1}^{\infty}\frac{\sin(2j)}{j}\operatorname{Ei}(-2j)
+
\left(\frac{\pi}{2}-1\right)(\gamma+\log \left(2\right))
+
\sum_{j=1}^{\infty}\frac{\sin(2j)}{j}\log \left(j\right),
\end{equation}

Finally, we shall evaluate the logarithmic sine series by Kummer’s formula. In other words, it remains to simplify the series

\[
\sum_{j=1}^{\infty}\frac{\sin(2j)}{j}\log \left(j\right).
\]
To do so, we apply the Kummer’s identity (Preliminaries, subsection \ref{subsec:3.3}) with $x=\frac{1}{\pi}$. Then we have $0<x<1$, and $2\pi j x = 2j$. Thus
\begin{align}
\sum_{j=1}^{\infty}\frac{\sin(2j)}{j}\log \left(j\right)
&=
\pi\log\left(\Gamma\!\left(\frac{1}{\pi}\right)\right)
-
\frac{\pi}{2}\log\!\left(\frac{\pi}{\sin \left(1\right)}\right)
-
\pi\left(\frac12-\frac{1}{\pi}\right)(\gamma+\log(2\pi))
\nonumber\\
&=
\pi\log\left(\Gamma\!\left(\frac{1}{\pi}\right)\right)
+
\left(1-\frac{\pi}{2}\right)(\gamma+\log(2\pi))
-
\frac{\pi}{2}\log\!\left(\frac{\pi}{\sin \left(1\right)}\right).
\label{eq:sine-log-Kummer}
\end{align}

Substituting \eqref{eq:sine-log-Kummer} into the relation \eqref{eq:K1-before-Kummer} yields
\begin{align}
K(1)
&=
-\sum_{j=1}^{\infty}\frac{\sin(2j)}{j}\operatorname{Ei}(-2j)
+\left(\frac{\pi}{2}-1\right)(\gamma+\log \left(2\right))
\nonumber\\
&\qquad
+\pi\log\left(\Gamma\!\left(\frac{1}{\pi}\right)\right)
+\left(1-\frac{\pi}{2}\right)(\gamma+\log(2\pi))
-\frac{\pi}{2}\log\!\left(\frac{\pi}{\sin \left(1\right)}\right),
\end{align}

Finally, we can combine the constant terms together. The coefficient of $\gamma$ is
\begin{equation}
\left(\frac{\pi}{2}-1\right)+\left(1-\frac{\pi}{2}\right)=0,
\end{equation}
so all Euler--Mascheroni terms cancel. For the logarithms of $2$, we have
\begin{equation}
\left(\frac{\pi}{2}-1\right)\log \left
(2\right)
+
\left(1-\frac{\pi}{2}\right)\log(2\pi)
=
\left(1-\frac{\pi}{2}\right)\log\left(\pi\right).
\end{equation}
Hence
\[
K(1)
=
-\sum_{j=1}^{\infty}\frac{\sin(2j)}{j}\operatorname{Ei}(-2j)
+\pi\log\left(\Gamma\!\left(\frac{1}{\pi}\right)\right)
+\left(1-\frac{\pi}{2}\right)\log\left(\pi\right)
-\frac{\pi}{2}\log\!\left(\frac{\pi}{\sin \left(1\right)}\right),
\]
which precisely corresponds to \eqref{eq:K1-final-compressed}. This completes the proof of Lemma \ref{lem.1}.
\end{proof}

\newpage
\vspace{1cm}
\hspace{-0.6cm}\textbf{Acknowledgments}

\vspace{0.1cm}
This paper was supported by the grant project Cartan supergeometries and Higher Cartan geometries No. 24-10887S, provided by the Czech Science Foundation. Additionally, there was local support from the project for Specific research in Mathematics MUNI/A/1457/2023 for doctoral students, provided by the Department of Mathematics and Statistics at Masaryk University.

I am grateful to Prof. Jan Slovák, DrSc. for valuable comments and recommended adjustments, and Dr. John M. Campbell for reviewing and endorsing this article. 

 \bibliographystyle{elsarticle-num} 
 \bibliography{cas-refs}

\end{document}